\documentclass[a4paper]{article}
\usepackage[UTF8, scheme=plain]{ctex}
\usepackage{amsmath,amsthm,amsfonts,mathrsfs}
\usepackage{amssymb}
\usepackage{extarrows}
\usepackage{tikz-cd}
\usepackage{xspace}
\usepackage{graphicx,booktabs}
\usepackage{subcaption}
\usepackage{paralist}
\usepackage{authblk}
\usepackage{hyperref}
\hypersetup{
    pdfmenubar=true,       % show Acrobat's menu?
    pdffitwindow=false,     % window fit to page when opened
    pdfstartview={FitH},    % fits the width of the page to the window
    colorlinks=true,       % false: boxed links; true: colored links
    linkcolor=red,          % color of internal links
    citecolor=red,        % color of links to bibliography
    filecolor=magenta,      % color of file links
    urlcolor=cyan,           % color of external links
}
\usepackage{acro}
\usepackage{cleveref}
\usepackage{algorithm}
\usepackage{algpseudocode}
\makeatletter
\def\BState{\State\hskip-\ALG@thistlm}
\makeatother
\Crefname{equation}{}{}
\usepackage{todonotes}
\usepackage{inputenc}
\usepackage{dsfont}
\usepackage{amsbsy}
\usepackage{enumerate}
\usepackage{mathtools}
\usepackage{xfrac}

\numberwithin{figure}{section}
\usepackage[export]{adjustbox} % export 选项让 includegraphics 支持对齐参数

\newtheorem{definition}{Definition}[section]
\newtheorem{theorem}{Theorem}[section]

\newtheorem{lemma}[theorem]{Lemma}
\newtheorem{assumption}[definition]{Assumption}
\newtheorem{myexample}{Example}[section]
\newtheorem{remark}[theorem]{Remark}
\numberwithin{figure}{section}

\numberwithin{equation}{section}

\newcommand{\rmd}{\mathrm{d}}
\newcommand{\bbR}{\mathbb{R}}
\newcommand{\caL}{\mathcal{L}}

\newcommand{\propod}{\mathcal{P}_{\mathrm{pod}}}

\newcommand{\keywords}[1]{%
  \par\vspace{0.5\baselineskip}%
  \noindent\textbf{Keywords:} #1%
  \par\vspace{0.5\baselineskip}%
}

\title{A Pseudo-time Data-Driven Framework for Model Reduction of Linear Operator Equations}

\author[1]{Zhentong Wei}
\author[2]{Tingen Xiong}
\author[2]{Wenlong Zhang}
\author[1,3]{Zhiwen Zhang}

\affil[1]{Materials Innovation Institute for Life Sciences and Energy (MILES), Shenzhen, People's Republic of China}
\affil[2]{Department of Mathematics \& National Center for Applied Mathematics Shenzhen, Southern University of Science and Technology (SUSTech), 1088 Xueyuan Boulevard, Shenzhen, Guangdong Province, People's Republic of China}
\affil[3]{Department of Mathematics, The University of Hong Kong, Pokfulam Road, Hong Kong SAR, People's Republic of China}

\date{}

\begin{document}
\begin{sloppypar}

\maketitle
% \begin{center}
% \textbf{Corresponding author:}
% \href{mailto:zhangwl@sustech.edu.cn}{zhangwl@sustech.edu.cn}; \href{mailto:zhangzw@hku.hk}
% {zhangzw@hku.hk};\\
% \textbf{Contributing authors:} \href{mailto:weizhentong@lsec.cc.ac.cn}{weizhentong@lsec.cc.ac.cn}
% \end{center}

\begin{abstract}
This paper proposes a novel Pseudo-time Proper Orthogonal Decomposition (POD) framework to enable model reduction for stationary problems lacking temporal snapshot data. By recasting static operator equations into a pseudo-dynamic evolution form, we artificially generate temporal data while preserving the system's intrinsic spectral properties. Mathematically, we rigorously prove the exponential convergence of the pseudo-time trajectory to the exact stationary solution. Furthermore, the approximation properties of the generated POD basis functions are rigorously established. Finally, the universality and accuracy of the framework are validated both theoretically and numerically across two representative settings: 
% linear algebraic equations, 
elliptic inverse source problems and Fredholm integral equations of the first kind.
\end{abstract}


% REQUIRED
% \begin{keywords}
%   Linear operator equation, Model reduction method,
% \end{keywords}

% REQUIRED
% \begin{MSCcodes}
% 68Q25, 68R10, 68U05
% \end{MSCcodes}

\section{Introduction}
Linear operator equations of the form $u = \mathcal{R} f$ constitute a cornerstone of modern applied mathematics \cite{engl1996regularization,hansen2010discrete}. These equations serve as the foundational mathematical models for a vast array of applications across science and engineering. This spectrum includes large-scale linear algebraic systems arising from the discretization of partial differential equations \cite{saad2003iterative}, geophysical inverse problems such as seismic tomography \cite{tarantola2005inverse}, pollution source control \cite{gorelick1983identifying}, and medical imaging techniques like X-ray computed tomography \cite{natterer2001mathematics}. Furthermore, they are indispensable in image processing for restoration \cite{vogel2002computational,bertero2021introduction} and optical remote sensing \cite{doicu2010numerical}. Solving these equations efficiently, particularly when the underlying Hilbert spaces are high-dimensional or the operators are ill-posed, remains a significant computational challenge.

To mitigate these computational demands, Model Order Reduction (MOR) has emerged as a powerful paradigm for constructing efficient, low-rank approximations of high-dimensional systems. Among various MOR techniques, the Proper Orthogonal Decomposition (POD) method is particularly prominent \cite{benner2015survey}. Originally derived in \cite{lumley1967structure} to describe coherent structures in turbulent flows, POD gained widespread attention only after the method of snapshots was introduced in \cite{sirovich1987turbulence}, providing a numerically tractable implementation for large-scale engineering datasets. This framework was further formalized in \cite{berkooz1993proper}, establishing POD as a cornerstone for capturing dominant spatial features in complex systems.

Over the past decades, POD has been successfully applied to a broad spectrum of time-evolution problems. Its versatility is demonstrated in fluid dynamics, where it has been used to reduce the Navier-Stokes equations \cite{kunisch2002galerkin} and the shallow water equations \cite{cstefuanescu2015pod}. In structural mechanics, POD has been employed to capture the dynamics of nonlinear vibrations \cite{lall2003structure}. The method's applicability extends to complex nonlinear PDEs, such as the FitzHugh-Nagumo equations \cite{chaturantabut2010nonlinear}, as well as the viscous G-equations \cite{gu2021error} and Hamilton–Jacobi–Bellman (HJB) equations \cite{kunisch2004hjb}. Furthermore, POD has become an indispensable tool in optimal control \cite{alla2013time}, parameter estimation \cite{willcox2006unsteady}, and Uncertainty Quantification (UQ) \cite{galbally2010non}. For a comprehensive review of these model reduction strategies, we refer the reader to \cite{benner2015survey,quarteroni2015reduced}.

Despite its success, a conceptual mismatch exists between POD and stationary operator equations. The POD framework is intrinsically designed to extract an optimal low-dimensional subspace from an ensemble of snapshots reflecting the system's transient dynamics \cite{kunisch2002galerkin}. Consequently, a significant challenge arises when addressing stationary linear equations: the static nature of the problem typically yields only a unique equilibrium state. This inherent paucity of data precludes the construction of a sufficiently diverse snapshot matrix, which is essential for identifying the dominant modes that span the solution space.

In this paper, we propose a pseudo-time data-driven POD framework that bridges the gap between stationary operator equations and dynamic reduction methods. A crucial distinction exists between our framework and the classical Showalter's method (or asymptotic 
regularization) \cite{engl1996regularization}. While Showalter's method introduces a temporal variable to regularize the solution of ill-posed inverse problems, our approach introduces pseudo-time as a mechanism to embed temporal diversity into the state data. Essentially, whereas the former focuses on the asymptotic stability of the solution process, our Pseudo-time POD framework focuses on generating a rich snapshot set from a stationary operator. 

To implement this, we recast the static equation $\mathcal{A}u =  f$, where $\mathcal{A}=\mathcal{R}^{-1}$, into a pseudo-dynamic evolution form:  
\begin{equation}
  \begin{aligned}
        \left\{\begin{array}{ll}
             \hat{u}_t + \mathcal{A} \hat{u} = f &\text { in }~ \Omega \times (0, T), \\[2mm] 
            \hat{u}(\cdot,0) = 0 &\text { in }~ \Omega,
            % u(x,0)= 0 &\text { in }~ \Omega.
        \end{array}\right.
    \end{aligned}  
\end{equation}
where $\Omega \subset \bbR^d$ ($d=1,2,3$) is a bounded domain.
Notably, by directly utilizing the original physical operator $\mathcal{A}$ rather than the normal operator $\mathcal{R}^*\mathcal{R}$, we preserve the intrinsic spectral properties of the system while artificially constructing the time-series snapshots required for POD. This strategy enables high-fidelity model reduction even for snapshot-deficient stationary systems. In practical scenarios, where 
$f$ is unknown, we leverage the observable data $m$ following the method in
\cite{zhang2025novel}, formulating the following system to extract the POD basis:
\begin{equation}
    \begin{aligned}
        \left\{\begin{array}{ll}
            \widetilde{u}_{t} + \mathcal{A} \widetilde{u} =  m &\text { in }~ \Omega \times (0, T), \\[2mm] 
            \widetilde{u}(\cdot,0) = 0 &\text { in }~ \Omega.\end{array}\right.
        \end{aligned}
\end{equation}

Mathematically, we establish a rigorous theoretical foundation for this pseudo-time data-driven POD framework. Under the standard assumptions that the operator 
$\mathcal{R}$ is compact, self-adjoint, and injective, we prove in \Cref{err_gradient_flow} that the solution of the pseudo-dynamic evolution equation converges asymptotically to the exact stationary solution. Furthermore, the convergence of the POD basis functions and their corresponding approximation properties are rigorously established in \Cref{err_ture_solution}. To demonstrate the universality and robustness of this framework, we apply it to two representative classes of problems:
% spanning both discrete and continuous settings: linear algebraic equations, 
 elliptic inverse source problems and Fredholm integral equations of the first kind.

The remainder of this paper is organized as follows. \Cref{sec_pod_frame} introduces the underlying linear operator equation and systematically develops the pseudo-time data-driven POD framework, detailing the corresponding theoretical error estimates. In \Cref{sec_pod_application}, we establish the broad applicability of this framework by applying it to the two representative problems mentioned above. \Cref{sec_pod_num} presents comprehensive numerical experiments to validate the accuracy and effectiveness of our approach. Finally, \Cref{sec_pod_concl} concludes the paper with a summary of our findings and a discussion of potential future research directions.

% \textbf{Contributions.} 

% \textbf{Outline.} 

\section{Linear Operator Equation}\label{sec_pod_frame}

% \section{Gradient Flow Method}

Consider the linear operator equation
\begin{equation}\label{eq_forward}
u = \mathcal{R} f,
\end{equation}
where $f \in X$ is the unknown to be recovered, $u \in Y$ is the observed data, and $\mathcal{R}: X \to Y$ is a bounded linear operator. The spaces $X$ and $Y$ are Hilbert spaces equipped with inner products $(\cdot,\cdot)_X, (\cdot,\cdot)_Y$ and their induced norms $\|\cdot\|_X, \|\cdot\|_Y$, respectively.

To recover $f$ from the observed data $u$, we seek a mapping from $Y$ to $X$. Given that $\mathcal{R}$ is injective (i.e.， $\ker(\mathcal{R}) = \{0\}$), the operator $\mathcal{R}$ is a bijection from $X$ onto its range $\operatorname{Ran}(\mathcal{R})$. Consequently, the inverse operator $\mathcal{A} := \mathcal{R}^{-1}$ is well-defined on $\operatorname{Ran}(\mathcal{R})$, enabling the unique recovery of $f$ via
\begin{equation}\label{eq_inverse}
    f=\mathcal{A} u.
\end{equation}

Here, our primary objective is to efficiently solve \eqref{eq_forward} using the Proper Orthogonal Decomposition (POD) method. We develop a two-step framework consisting of a dynamic formulation via pseudo-time embedding (\Cref{sec_pseudo_time}) followed by the construction of a reduced-order model based on data-driven snapshots (\Cref{sec_pod_construct}).

\subsection{Pseudo-Time Scheme}\label{sec_pseudo_time}
The application of the POD method fundamentally relies on a collection of data snapshots to extract a low-dimensional subspace. However, the original problem \eqref{eq_forward} only yields a single stationary state, which is insufficient to provide the rich dataset required for POD. To circumvent this limitation, we embed the problem \cref{eq_inverse} into a dynamic framework. By introducing a pseudo-time variable $t$ and treating $u$ as a time-dependent state, we construct the following evolution equation to model a temporal trajectory:
% we construct the following conceptual evolution equation to theoretically model a temporal trajectory:
\begin{equation}\label{eq_evolution}
  \begin{aligned}
        \left\{\begin{array}{ll}
             \hat{u}_t + \mathcal{A} \hat{u} = f &\text { in }~ \Omega \times (0, T), \\[2mm] 
            \hat{u}(\cdot,0) = 0 &\text { in }~ \Omega.
            % u(x,0)= 0 &\text { in }~ \Omega.
        \end{array}\right.
    \end{aligned}  
\end{equation}

To rigorously characterize the asymptotic behavior of the solution $\hat{u}(\cdot,t)$ to the evolution equation \eqref{eq_evolution} and its convergence to the solution $u$ of the equation \eqref{eq_inverse} for a given $f$, we assume that $X=Y$ is a real or complex Hilbert space and impose the following assumptions on the operator $\mathcal{R}$.
% To begin our study, we need to make certain assumptions about the operator $\mathcal{R}$.
\begin{assumption}\label[assumption]{compact_self_adjoint}
    The operator $\mathcal{R}\in \mathcal{L}(X)$, where $\mathcal{L}(X)$ denotes the space of bounded linear operators on $X$, is compact, self-adjoint, and injective. Its spectral decomposition is governed by the eigenvalue problem:
    \begin{equation}\label{eq_eig_problem}
        \mathcal{R} \phi = \frac{1}{\mu} \phi \quad \text{with} \quad \phi_{\partial \Omega} =0.
    \end{equation}
    There exists a countable sequence of positive eigenvalues $1/\mu_1\geq 1/\mu_2\geq \cdots > 0$ (with $\mu_k \rightarrow \infty$ as $k\rightarrow \infty$), and the corresponding eigenfunctions
    $\{ \phi_k\}_{k=1}^{\infty}$ form an orthonormal basis of $X$. Consequently, the inverse operator 
    $\mathcal{A}=\mathcal{R}^{-1}$ possesses eigenvalues $0< \mu_1 \leq \mu_2 \leq \cdots$, sharing the same eigenfunctions $\{ \phi_k\}_{k=1}^{\infty}$.    
\end{assumption}

Following \Cref{compact_self_adjoint}, the $f$ and $u$ in problem \cref{eq_forward} can be expressed in terms of the eigenfunctions $\{ \phi_k\}_{k=1}^{\infty}$ of \cref{eq_eig_problem} as
 \begin{equation}\label{eq_expansion}
     f=\sum_{k=1}^{\infty} f_k \phi_k \quad \text{and}\quad u=\sum_{k=1}^{\infty} u_k \phi_k,
 \end{equation}
where $f_k=(f,\phi_k)_{X}$ and $u_k=(u,\phi_k)_{X}$. Substituting these expansions into \cref{eq_forward}, we obtain
\begin{equation}\label{eq_forward_solution}
    u= \sum_{k=1}^{\infty} \frac{f_k}{\mu_k} \phi_k.
\end{equation}
Analogously, for the evolution equation \cref{eq_evolution}, we expand the time-dependent state $
\hat{u}(\cdot,t)\in X$ in the same orthonormal basis:
\begin{equation}
    \hat{u}(\cdot, t)=\sum_{k=1}^{\infty}\hat{u}_k(t)\phi_k,
\end{equation}
where $\hat{u}_k(t)=(\hat{u}(\cdot,t),\phi_k)_{X}$. By substituting this into \cref{eq_evolution} and solving the corresponding decoupled ODEs with the initial condition $\hat{u}_k(0)=0$, we arrive at
\begin{equation}\label{eq_evolution_solution}
    \hat{u}(\cdot,t)=\sum_{k=1}^{\infty}\frac{1}{\mu_k}(1-e^{-\mu_k t})f_k\phi_k
\end{equation}

Comparing \cref{eq_evolution_solution} with \cref{eq_forward_solution}, it is mathematically evident that the exponential term $e^{-\mu_k t}$ decays to zero as $t\rightarrow \infty$. This implies that the dynamic trajectory $\hat{u}(\cdot,t)$ asymptotically converges to the stationary solution $u$ of \cref{eq_forward}. In this process, the evolution equation \cref{eq_evolution} acts as a temporal relaxation towards the equilibrium state defined by $\mathcal{A}u = f$. This asymptotic convergence is rigorously quantified by the following theorem.

\begin{theorem}\label{err_gradient_flow}
    Let $u$ be the solution to the problem \cref{eq_inverse} and $\hat{u}(\cdot,t)$ be the solution to the problem \cref{eq_evolution} for a given $f\in X$. Then, for any $T>0$, the following error estimates hold:
    \begin{equation}
        \| \hat{u}(\cdot,T) - u \|_{X} \leq \frac{e^{-\mu_1 T}}{\mu_1} \| f \|_{X}\quad \text{and}\quad \|\hat{u}_t(\cdot,T)|\vert_X \leq e^{-\mu_1 T}\|f|\vert_X.
    \end{equation}
\end{theorem}
\begin{proof}
    First, using \cref{eq_evolution_solution} and \cref{eq_forward_solution}, we have
    \begin{equation}
            % \begin{aligned}
            % \|e(T)|\vert_X&=
            \|\hat{u}(\cdot,T)-u\|_X =
            \Big \|\sum_{k=1}^\infty\frac{1}{\mu_k}(1-e^{-\mu_k T})f_k\phi_k-\sum_{k=1}^\infty\frac{f_k}{\mu_k}\phi_k\Big \|_X
            =\Big \|\sum_{k=1}^\infty\frac{e^{-\mu_k T}}{\mu_k}f_k\phi_k\Big \|_X.
            % &\leq  \frac{e^{-\mu_1 T}}{\mu_1}\Big \|\sum_{k=1}^\infty f_k\phi_k\Big \|_X\\[2mm]
            % &\leq \frac{e^{-\mu_1 T}}{\mu_1}\|f|\vert_X,
        % \end{aligned}
    \end{equation}
    Since $g(x) = \frac{e^{-xT}}{x}$ is strictly decreasing for $x>0$, we have $\frac{e^{-\mu_k T}}{\mu_k} \le \frac{e^{-\mu_1 T}}{\mu_1}$ for all $k\geq 1$. By Parseval's identity, it follows that
    \begin{equation}
        \|\hat{u}(\cdot,T)-u\|_X^2 = \sum_{k=1}^\infty \frac{e^{-2\mu_k T}}{\mu_k^2} |f_k|^2 \leq  \frac{e^{-2\mu_1 T}}{\mu_1^2} \sum_{k=1}^\infty |f_k|^2 = \frac{e^{-2\mu_1 T}}{\mu_1^2} \|f\|_X^2
    \end{equation}
    
    Next, substituting $\hat{u}_t = f - \mathcal{A}\hat{u}$ into the error equation and noting $f = \mathcal{A}u$, we obtain $\hat{u}_t = \mathcal{A}(u - \hat{u}(\cdot,t))$. Thus,
    % Next, subtracting $f=\mathcal{A} u^*$ from equation \cref{eq_evolution} yields the error estimate for $\|u_t(\cdot,T)|\vert_X$:
    \begin{equation}
        \begin{aligned}
            \|\hat{u}_t(\cdot,T)\|_X=\|\mathcal{A}\big(\hat{u}(\cdot,T)-u^*\big)|\vert_X
            =\Big \|\sum_{k=1}^\infty e^{-\mu_k T}f_k\phi_k \Big \|_X.
            % =\Big \|\sum_{k=1}^\infty e^{-\mu_k T} f_k\phi_k\Big \|_X
            % \leq e^{-\mu_1 T}\|f\|_X.
        \end{aligned}
    \end{equation}
    Applying Parseval's identity again and using $e^{-\mu_k T} \le e^{-\mu_1 T}$, we obtain
    \begin{equation}
        \begin{aligned}
            \|\hat{u}_t(\cdot, T)\|_X^2 = \sum_{k=1}^\infty e^{-2\mu_k T} |f_k|^2 \leq e^{-2\mu_1 T} \sum_{k=1}^\infty |f_k|^2 = e^{-2\mu_1 T} \|f\|_X^2.
        \end{aligned}
    \end{equation}
    This completes the proof.
\end{proof}

\subsection{Construction of the POD Basis Functions}\label{sec_pod_construct}
As established in \cref{sec_pseudo_time}, the primary motivation for the pseudo-time embedding is to construct a time-evolving trajectory $\{\hat{u}(\cdot, t) : t \in [0, T]\}$ that captures the essential features of the solution $u$. In the context of solving the operator equation $u = \mathcal{R}f$, a direct high-dimensional computation can be computationally prohibitive, particularly when the system is large-scale. By generating a temporal sequence of snapshots along the relaxation process towards the steady state $u$, we can apply the POD method to identify a low-dimensional subspace that effectively approximates the solution space. This approach allows us to represent the solution $u$ as a linear combination of a few dominant basis functions, thereby reducing the dimensionality of the problem while preserving its fundamental structure.

However, a critical challenge in constructing the POD basis is that the value of $f$  is unknown, which prevents the direct generation of the temporal snapshots $\{\hat{u}(\cdot, t): t \in [0, T]\}$ required for the POD method. To overcome this, we employ the adjoint POD approach proposed in \cite{zhang2025novel}. Specifically, we utilize the available measurement $m$ as a surrogate for the unknown $f$ to drive the evolution process. We then generate the required snapshots by solving the following data-driven evolution equation:

\begin{equation}\label{gradient_flow_adjoint}
    \begin{aligned}
        \left\{\begin{array}{ll}
            \widetilde{u}_{t} + \mathcal{A} \widetilde{u} =  m &\text { in }~ \Omega \times (0, T), \\[2mm] 
            \widetilde{u}(\cdot,0) = 0 &\text { in }~ \Omega.\end{array}\right.
        \end{aligned}
\end{equation}
By using the measurement $m$ to drive the evolution, we construct a surrogate trajectory $\{\widetilde{u}(\cdot, t): t \in [0, T]\}$ that captures the dominant modes of the solution space despite the lack of knowledge regarding $f$. The theoretical justification for this approach is detailed in \Cref{sec_conv_pod}.

To construct the discrete snapshot set，we partition the time interval uniformly with a step size $\Delta t = T/ M$, yielding discrete time instances $t_k = k \Delta t$. The snapshot set is formed by collecting both the state solutions and their difference quotients: let $\widetilde{y}_k = \widetilde{u}(\cdot, t_{k-1})$, $k=1,\ldots,M+1$, and 
$\widetilde{y}_k = \overline{\partial} \widetilde{u}(\cdot, t_{k-M-1})$, $k =M+2,\ldots, 2M+1$, where $\overline{\partial} \widetilde{u}(\cdot, t_{k})= \frac{\widetilde{u}(\cdot,t_k) - \widetilde{u}(\cdot,t_{k-1})}{\Delta t}$. 

With the snapshot set $\{\tilde{y}_k\}_{k=1}^{2M+1}$ prepared, we proceed to extract the low-dimensional POD subspace $V_{pod}$ as follows.

% Once this snapshot set $\{\widetilde{y}_k\}_{k=1}^{2M+1}$ is obtained, we apply the POD method to extract a low-dimensional subspace, denoted by $V_{pod}$. The specific implementation details proceed as follows.

Let $Y= \operatorname{span}\{\widetilde{y}_1,\ldots,\widetilde{y}_{2M+1}\}$ with $\operatorname{dim} Y =N$. We define the correlation matrix $K =(K_{i,j}) \in \bbR^{(2M+1)\times (2M+1)}$ by
\[
    K_{ij}=\frac{1}{2M+1}( \widetilde y_i,\widetilde y_j )_{X},\quad i,j=1,\dots,2M+1.
\]
By definition, the matrix $K$ is positive semi-definite and has rank $N$. We then solve the eigenvalue problem
\begin{equation}
    K \nu=\lambda \nu. 
\end{equation}
The resulting eigenvalues are sorted in descending order such that $\lambda_1\geq\lambda_2\geq\ldots\geq\lambda_{2M+1}\geq 0$, with the corresponding orthonormal eigenvectors denoted by $\nu_1, \nu_2, \dots,\nu_{2M+1}$. 

% Finally, we truncate the spectrum by retaining only the first $N_{pod}$ dominant eigenvalues, ensuring that $\lambda_1 \geq \dots \geq \lambda_{N_{pod}} > 0$. Using these retained eigenpairs, the corresponding POD basis functions (also known as POD modes) are computed by:
% \begin{equation}
%     \psi_k = \frac{1}{\sqrt{\lambda_k}} \sum_{j=1}^{2M+1} (\nu_k)_j \tilde{y}_j, \quad 1 \le k \le N_{pod}. \label{eq:pod_basis}
% \end{equation}
% Consequently, the reduced-order POD subspace is defined as the linear span of these basis functions, i.e., $V_{pod} = \text{span}\{\psi_1, \dots, \psi_{N_{pod}}\}$.

% Furthermore, using these eigenpairs, the POD basis functions (also known as POD modes) are computed as
% \begin{equation}
%     \psi_k=\frac{1}{\sqrt{\lambda_k}}\sum_{j=1}^{2M+1}(\nu_k)_j \widetilde{y}_j,\qquad 1\leq k\leq d
% \end{equation}
% Finally, by truncating the spectrum and retaining only the first $N_{pod}$ dominant modes ($N_{pod} \leq d$), we define the reduced-order POD subspace spanned by these basis functions as
% $V_{pod}=\operatorname{span}\{\psi_1,\cdots, \psi_{N_{pod}}\}$.

According to \cite[proposition 1]{kunisch2001galerkin}, if we denote the positive eigenvalues of 
$K$ by $\lambda_1 \geq \cdots \geq \lambda_{N} > 0$ and the associated eigenvectors by $v_1, \ldots, v_{N} \in \mathbb{R}^{2M+1}$, then a POD basis of rank $l \leq N$ is given by
\begin{equation}
\psi_k = \frac{1}{\sqrt{\lambda_k}} \sum_{j=1}^{2M+1} (v_k)_j \widetilde{y}_j,\quad k=1,\ldots,l,
\end{equation}
where $(v_k)_j$ is the $j$-th component of the eigenvector $v_k$. This construction satisfies the error identity
\begin{equation}
    \begin{split}
        &\frac{1}{2M+1} \sum_{j=1}^{2M+1} \left\| \widetilde{y}_j - \sum_{k=1}^{l} (\widetilde{y}_j, \psi_k)_{X} \psi_k \right\|_{X}^2 
        % &+ \frac{1}{2M+1} \sum_{j=1}^{M} \left\|  \widetilde{y}_j - \sum_{k=1}^{l} (\bar{\partial} \widetilde{y}(t_j), \psi_k)_{X} \psi_k \right\|_{X}^2 
        = \sum_{k=l+1}^{N} \lambda_k.
    \end{split}
\end{equation}
% \begin{proposition}{\cite[proposition 1]{kunisch2001galerkin}}
% Let $\lambda_1 \geq \cdots \geq \lambda_{d} > 0$ denote the positive eigenvalues of $K$ and $v_1, \ldots, v_{d} \in \mathbb{R}^{2M+1}$ the associated eigenvectors. Then a POD basis of rank $l \leq d$ is given by
% \begin{equation}
% \psi_k = \frac{1}{\sqrt{\lambda_k}} \sum_{j=1}^{2M+1} (v_k)_j \widetilde{y}_j,
% \end{equation}
% where $(v_k)_j$ is the $j$-th component of the eigenvector $v_k$. Moreover, we have the error formula
% \begin{equation}
%     \begin{split}
%         &\frac{1}{2M+1} \sum_{j=1}^{2M+1} \left\| \widetilde{y}_j - \sum_{k=1}^{l} (\widetilde{y}_j, \psi_k)_{X} \psi_k \right\|_{X}^2 
%         % &+ \frac{1}{2M+1} \sum_{j=1}^{M} \left\|  \widetilde{y}_j - \sum_{k=1}^{l} (\bar{\partial} \widetilde{y}(t_j), \psi_k)_{X} \psi_k \right\|_{X}^2 
%         = \sum_{k=l+1}^{d} \lambda_k.
%     \end{split}
% \end{equation}
% \end{proposition}
Finally, by retaining the first $N_{pod}$ dominant modes (i.e., setting $l=N_{pod}$), we define the reduced-order POD subspace as $V_{pod}=\operatorname{span}\{\psi_1,\ldots, \psi_{N_{pod}}\}$.

\subsection{Convergence of the POD Method}\label{sec_conv_pod}
To analyze the approximation properties of the POD basis derived from $\widetilde{u}$, we assume that $f$ admits the following finite-dimensional representation:
\begin{equation}\label{f_app}
    f_{app}=\sum_{k=1}^L f_k\phi_k.
\end{equation}
Under this assumption, \cref{eq_forward_solution} and \cref{eq_evolution_solution} are rewritten as 
\begin{equation}
    % \begin{aligned}
        u=\sum_{k=1}^L\frac{1}{\mu_k}f_k\phi_k\quad \text{and}\quad
        \hat{u}(\cdot,t)=\sum_{k=1}^L\frac{1}{\mu_k}(1-e^{-\mu_kt})f_k\phi_k.
    % \end{aligned}
\end{equation}
In the noise-free setting where $m=u$, the computed state $\widetilde{u}$ satisfies
\begin{equation}
    \widetilde{u}(\cdot,t)=\sum_{k=1}^L\frac{1}{\mu_k^2}(1-e^{-\mu_k t})f_k\phi_k.
\end{equation}

Subsequently, let $\mathbf{x}= (x_1,\ldots,x_N)^{\top}$ denote the vector of finite element nodes in $\Omega$. We construct the discrete snapshot matrices for the states and their temporal derivatives as follows:
% \begin{equation}
%     \begin{aligned}
%         \textbf{A} = [y_2,\ldots,y_{M+1}], \textbf{A}_1=[y_{M+2},\cdots,y_{2M+1}]
%     \end{aligned}
% \end{equation}
\begin{align}
    \textbf{A}_1 &= [\hat{y}_2,\ldots,\hat{y}_{M+1}],\ \widetilde{\textbf{A}}_1 = [\widetilde{y}_2,\ldots,\widetilde{y}_{M+1}],\\[2mm]
    \textbf{A}_2&=[\hat{y}_{M+2},\cdots,\hat{y}_{2M+1}],\
    \widetilde{\textbf{A}}_2=[\widetilde{y}_{M+2},\cdots,\widetilde{y}_{2M+1}],
\end{align}
where $\hat{y}_j$ and $\widetilde{y}_{j}$ are evaluations of the respective continuous functions at $\mathbf{x}$.
% the column vectors are obtained by evaluating the continuous functions at the spatial nodes $\mathbf{x}$:
% \begin{align}
%     y_j &= u(\mathbf{x},t_{j-1}),\ \widetilde{y}_j = \widetilde{u}(\mathbf{x},t_{j-1}),\ \text{for}\ j=2,\ldots,M+1,\\[2mm]
%     y_{j} &= \overline{\partial} u(\mathbf{x}, t_{j-M-1}),\ \widetilde{y}_{j} = \overline{\partial} \widetilde{u}(\mathbf{x}, t_{j-M-1}),\ \text{for}\ j=M+2,\ldots,2M+1.
% \end{align}
To facilitate analysis, we introduce the spatial basis matrix  
$\Phi = [\phi_1 (x),\ldots, \phi_L (x)]$, the coefficient matrices $\mathrm{F} = \mathrm{diag}(f_1,\ldots, f_L)$ and $\mathrm{D} = \mathrm{diag}(\frac{1}{\mu_1},\ldots,\frac{1}{\mu_L})$, and the temporal matrix $\mathrm{J} \in \bbR^{L\times M}$ with $\mathrm{J}_{i,j} = \frac{1}{\mu_i}(1-e^{-\mu_i t_j})$.
% \begin{equation}
%     \mathrm{J}_{i,j} = \frac{1}{\mu_i}(1-e^{-\mu_i t_j})\quad \text{for } i = 1, \dots, L \text{ and } j = 1, \dots, M.
% \end{equation}
Letting $\mathrm{U} \in \bbR^{M\times M}$ denote the finite difference matrix with $\mathrm{U}_{i,i} = \frac{1}{\Delta t}$ and $\mathrm{U}_{i,i+1} = -\frac{1}{\Delta t}$, the snapshot matrices can be factorized as:
\begin{equation}
    \textbf{A}_1 = \Phi \mathrm{F} \mathrm{J},\quad \widetilde{\textbf{A}}_1 = \Phi \mathrm{D} \mathrm{F} \mathrm{J},\quad \textbf{A}_2 = \Phi \mathrm{F} \mathrm{J} \mathrm{U},\quad \widetilde{\textbf{A}}_2 = \Phi \mathrm{D} \mathrm{F} \mathrm{J} \mathrm{U}
\end{equation}

% Next, let $\mathbf{x}= (x_1,\ldots,x_N)^{\top}$ denote the vector of finite element nodes in $\Omega$. 
% We define the basis matrix
% \begin{equation}
%     \Phi = [\phi_1 (\mathbf{x}),\ldots, \phi_L (\mathbf{x})],
% \end{equation}
% where $\phi_i (\mathbf{x}) = \big(\phi_i (x_1),\ldots,\phi_i (x_N)\big)^{\top}$. 

% We denote $A=(y_2,\cdots,y_{M+1})$, $\widetilde{A}=(\widetilde{y}_2,\cdots,\widetilde{y}_{M+1})$, $A_1=(y_{M+2},\cdots,y_{2M+1})$, and $\widetilde{A}_1=(\widetilde{y}_{M+2},\cdots,\widetilde{y}_{2M+1})$. Using the formulations of $v$ and $\widetilde{v}$, we can represent
% the matrices $A$ and $\widetilde{A}$ as follows:
% \begin{equation}
%     A=\Phi F J,\qquad \widetilde{A}=\Phi D F J,\qquad A_1=\Phi FJ\widetilde{L},\qquad \widetilde{A}_1=\Phi DFJ\widetilde{L},
% \end{equation}
% where $F=\operatorname{diag}(f_1,\cdots,f_L)$, $D=\operatorname{diag}(\frac{1}{\mu_1},\cdots,\frac{1}{\mu_L})$,  $J(i,j)=\frac{1}{\mu_i}(1-e^{-\mu_i t_j})$  , and $\widetilde{L}(i,i)=\frac{1}{\Delta t}$, $\widetilde{L}(i-1,i)=-\frac{1}{\Delta t}$ for all $i=1,\cdots,M$.

Following \cite[Lemma 2.2]{zhang2025novel},
$\mathrm{J}$ has full row rank when $L\leq M$. Consequently, there exist invertible matrices $\mathrm{P}$ and $\widetilde{\mathrm{P}}$ such that $\Phi \mathrm{D} \mathrm{F} \mathrm{J} \mathrm{P}=\Phi \mathrm{F} \mathrm{J}$, and $\Phi \mathrm{D} \mathrm{F} \mathrm{J} \mathrm{U}=\Phi \mathrm{F} \mathrm{J} \mathrm{U}\widetilde{\mathrm{P}}$. This equivalence implies that the pseudo-time snapshot matrices $\textbf{A}_1,\textbf{A}_2$ and the computed snapshot matrices $\widetilde{\textbf{A}}_1, \widetilde{\textbf{A}}_2$ share identical column  spaces. Based on this, \Cref{error_parabolic_adjoint_pod_1} establishes the projection error bound for the pseudo-time states $\hat{y}_i$ onto the adjoint-POD space $V_{pod}$.

\begin{theorem}\label{error_parabolic_adjoint_pod_1}
    % Suppose that a sufficient number of snapshots are available such that 
    Assume $L\leq M$. Then, the projection error of the pseudo-time states $\hat{y}_i$ onto the adjoint-POD space $V_{pod}$ satisfies 
     % We can derive an approximation error bound if a sufficient number of snapshots are available, i.e., $L\leq M$. In this case, the following error bound holds:
    \begin{equation}
        \frac{1}{2M+1}\sum_{i=1}^{2M+1}\| \hat{y}_i - \propod \hat{y}_i\|_{X}^2 \leq C \mu_L^{2} \rho,
    \end{equation}
    where $\propod$ denotes the orthogonal projection operator onto the adjoint-POD space $V_{pod}$, and $\rho = \sum_{k=N_{pod}+1}^{2M+1}\lambda_k$
    represents the truncation error, which is determined by the decay rate of the eigenvalues $\lambda_k$ of the correlation matrix $K$.
    % where $\propod$ is the orthogonal projection operator onto the adjoint-POD space $\mathrm{span}\{\psi_1,\ldots,\psi_{N_{pod}}\}$ and $\rho = \sum_{k=N_{pod}+1}^{2M+1}\lambda_k$ is a parameter that depends on the decay speed of the eigenvalues of the correlation matrix.
\end{theorem}
\begin{proof}
    For simplicity, we assume $L = M$ in the following proof; the case $L < M$ follows analogously. From the preceding discussion, we know that the matrices $\Phi$ and $\mathrm{J}$ are invertible, and there exists an invertible matrix $\mathrm{P}$ such that
    % In the following proof, we consider $L=M$ for simplicity. For the case $L<M$, the proof is similar. From the above discussion, we know that matrices $\Phi$ and $J$ are  invertible, and there exists an invertible matrix $P$ such that:
    \begin{equation}\label{snapshot}
        \Phi \mathrm{D} \mathrm{F} \mathrm{J} \mathrm{P}=\Phi \mathrm{F} \mathrm{J}.
    \end{equation}
    It follows that $\mathrm{P}=\mathrm{J}^{-1}\mathrm{D}^{-1}\mathrm{J}$. Furthermore, we have $\hat{y}_j=\sum_{i=1}^L \mathrm{P}_{ij}\widetilde{y}_i$. Using the Cauchy-Schwarz inequality, we can show that for any $1\leq j\leq L$,
    \begin{equation}\label{eq_snapshot}
    \begin{aligned}
        \frac{1}{2M+1}\sum_{j=1}^{M}\|\hat{y}_{j+1}-\propod \hat{y}_{j+1}\|_X^2
        &\leq \frac{1}{2M+1}\sum_{j=1}^{M}\sum_{i=1}^L \mathrm{P}_{ij}^2\|\widetilde{y}_{i+1}-\propod \widetilde{y}_{i+1}\|_X^2\\[2mm]
        &\leq \|\mathrm{P}|\vert_F^2\frac{1}{2M+1}\sum_{i=2}^{M+1}\|\widetilde{y}_i-\propod\widetilde{y}_i\|_X^2\\[2mm]
        &\leq\|\mathrm{P}\|^2_F \rho.
    \end{aligned}
    \end{equation}
It remains to estimate the Frobenius norm of $\mathrm{P}$. Since $\mathrm{P}=\mathrm{J}^{-1}\mathrm{D}^{-1}\mathrm{J}$, we can define a matrix norm $\|\mathrm{P}|\vert_d=\|\mathrm{D}^{-1}|\vert_2$. It is straightforward to verify that $\|\cdot|\vert_d$ is a matrix norm. Thus, we have
\begin{equation}\label{eq_p_matrix_norm}
    \|\mathrm{P}\|_F\leq C \|\mathrm{P}\|_d=C\|\mathrm{D}^{-1}\|_2\leq C\mu_L.
\end{equation}

By a similar argument, analyzing the matrices $\textbf{A}_2$ and $\widetilde{\textbf{A}}_2$, we obtain $\widetilde{\mathrm{P}}=\mathrm{U}^{-1}\mathrm{J}^{-1}\mathrm{D}^{-1}\mathrm{J}\mathrm{U}$.
Using the Cauchy-Schwarz inequality again, we can show that for any
$1\leq j\leq L$,
\begin{equation}\label{snapshot_difference}
    \begin{aligned}
        \frac{1}{2M+1}\sum_{j=M+2}^{2M+1}\|\hat{y}_j-\propod \hat{y}_j\|_X^2
        &\leq\frac{1}{2M+1}\sum_{j=1}^{M}\sum_{i=1}^L\widetilde{P}_{ij}^2\|\widetilde{y}_{M+1+i}-\propod\widetilde{y}_{M+1+i}\|_X^2\\[2mm]
        &\leq \|\widetilde{\mathrm{P}}\|_F^2\frac{1}{2M+1}\sum_{i=M+2}^{2M+1}\|\widetilde{y}_i-\propod\widetilde{y}_i\|_X^2\\[2mm]
        &\leq \|\widetilde{\mathrm{P}}|\vert_F^2\rho.
    \end{aligned}
\end{equation}
Following the same logic as in \cref{eq_p_matrix_norm}, it can be shown that $\|\widetilde{\mathrm{P}}|\vert_F\leq C \mu_L$.

Finally, by combining the estimates \cref{eq_snapshot} and \cref{snapshot_difference}, and incorporating the initial condition $\hat{y}_1=0$, we can conclude that the total projection error satisfies:
\begin{equation}
    \frac{1}{2M+1}\sum_{i=1}^{2M+1}\| \hat{y}_i - \propod \hat{y}_i\|_{X}^2 \leq C \mu_L^{2} \rho.
\end{equation}
This completes the proof.
\end{proof}

% In \cite{kunisch2001galerkin}, the authors point out that using only solution snapshots mainly captures the spatial states of a parabolic system, but may not accurately represent its temporal evolution. Since the reduced Galerkin model must also approximate the time derivative in the governing equation, the snapshot set is enriched by adding temporal difference quotients. These difference terms contain information on changes between successive time levels, so the resulting POD basis better captures the dynamics of the problem and leads to sharper error estimates and improved stability. Now we present another advantage of adding temporal difference quotients: pointwise POD error \cite{koc2021optimal}.

In \cite{kunisch2001galerkin}, the authors observe that using only solution snapshots primarily captures the spatial features of a parabolic system, often failing to accurately represent its temporal evolution. As the reduced-order Galerkin model must also approximate the time derivative, the snapshot set is typically enriched with temporal difference quotients. These terms incorporate information regarding the transitions between successive time levels, allowing the resulting POD basis to better characterize the underlying dynamics. This approach not only yields sharper error estimates and enhanced stability but also provides another significant advantage: the attainment of a pointwise POD error bound \cite{koc2021optimal}.

\begin{lemma}
    Let $T>0$ and $\Delta t=\frac{T}{M}$. Then, the pointwise POD error satisfies
    % Then, we have pointwise POD error: 
    \begin{equation}
        \max_{1\leq j\leq M+1}\|\hat{y}_j-\propod \hat{y}_j|\vert_X^2\leq C_1 C\mu_L^{2}\rho,
    \end{equation}
    where  $C_1=6\max\{1,T^2\}$.
\end{lemma}
\begin{proof}
    Define $z_j=\hat{y}_j-\propod \hat{y}_j$ for $\forall j=1,\dots,M+1$, and the discrete time derivative $\bar\partial z_j=\frac{z_{j+1}-z_{j}}{\Delta t}$ for $\forall j=1,\dots,M$. By expressing $z_k-z_l = \Delta t\sum_{n=l+1}^{k}\bar{\partial}z_n$, and applying the triangle inequality, we have
    \begin{equation}\label{eq_z_k}
        \begin{aligned}
            \|z_k|\vert_X&\leq \|z_l|\vert_X+\sum_{n=1}^M\Delta t^{\frac{1}{2}}(\Delta t^{\frac{1}{2}}\|\bar{\partial}z_n|\vert_X)\\[2mm]
            &\leq \|z_l|\vert_X+T^{\frac{1}{2}}(\sum_{n=1}^{M}\Delta t\|\bar{\partial}z_n|\vert_X^2)^{\frac{1}{2}},
        \end{aligned} 
    \end{equation}
    where the second inequality follows from the Cauchy-Schwarz inequality and $\sum_{n=1}^M \Delta t = T$. Let $l$ be the index such that $\|z_l|\vert_X=\min_{1\leq n\leq M+1}\|z_n|\vert_X$. Then,
    \begin{equation}\label{eq_z_l}
            \|z_l|\vert_X
            % = \frac{1}{M+1}\sum_{n=1}^{M+1}\|z_l|\vert_X 
            = \frac{1}{M+1}\sum_{n=2}^{M+1}\|z_l|\vert_X
            \leq \frac{1}{T}\sum_{n=2}^{M+1}\Delta t\|z_n|\vert_X\leq T^{-\frac{1}{2}}(\sum_{n=2}^{M+1}\Delta t\|z_n|\vert_X^2)^{\frac{1}{2}},
    \end{equation}
    where we used $\frac{1}{M+1} < \frac{1}{M} = T^{-1}\Delta t$, alongside the Cauchy-Schwarz inequality. Substituting \cref{eq_z_l} into \cref{eq_z_k} and utilizing $(a+b)^2 \leq 2(a^2 + b^2)$, we obtain
    \begin{equation}\label{eq_z_k_2}
        \|z_k|\vert_X^2\leq \frac{2}{T}\sum_{n=2}^{M+1}\Delta t\|z_n|\vert_X^2+ 2T\sum_{n=1}^M\Delta t \|\bar{\partial}z_n|\vert^2_X.
    \end{equation}
    Applying \Cref{error_parabolic_adjoint_pod_1} and the identity $\Delta t \leq \frac{3T}{2M + 1}$, we obtain
    
    % In \Cref{error_parabolic_adjoint_pod_1}, we have already obtained the upper bounds for the first and second terms on the right-hand side of the above inequality.
    % Since squaring both sides, and using the inequalities $(a+b)^2\leq 2(a^2+b^2)$, $\Delta t=(2T+\Delta t)/(2M+1)\leq 3T/(2M+1)$, we obtain the result:
    \begin{equation}
         \max_{1\leq j\leq M+1}\|z_j|\vert^2_X=\max_{1\leq j\leq M+1}\|\hat{y}_j-\propod \hat{y}_j|\vert_X^2\leq C_1 C\mu_L^{2}\rho,
    \end{equation}
    where $C_1=6\max\{1,T^2\}$. This completes the proof.
\end{proof}

Based on the preceding estimate, we now characterize the error of the proposed method. Let $a(\cdot,\cdot): \mathcal{D}(\mathcal{A}^{\frac{1}{2}}) \times \mathcal{D}(\mathcal{A}^{\frac{1}{2}}) \rightarrow \bbR$ be the bilinear form defined by 
\begin{equation}
    a(u,v)=(\mathcal{A}u,v)_X.
\end{equation}
Given that $\mathcal{A}$ is an unbounded self-adjoint operator, this bilinear form is well-defined. We further introduce the inner product $(u,v)_{\mathcal{A}}:=(\mathcal{A}u,v)_X$ and its induced norm $\|u|\vert_{\mathcal{A}} := \sqrt{(\mathcal{A}u,u)_X}$. For any $u\in \mathcal{D}(\mathcal{A}^{\frac{1}{2}}) $,
the following norm equivalence holds:
\begin{equation}\label{norm_equi}
\sqrt{\mu}_1 \lVert u \rVert_X \le \lVert u \rVert_{\mathcal{A}} \le \sqrt{\|\mathrm{V}\|_2} \lVert u \rVert_X,
\end{equation}
where $\mathrm{V} \in \bbR^{N\times N}$ with $\mathrm{V}_{ij}=(\psi_i,\psi_j)_{\mathcal{A}}$. 

Within this framework, we seek the approximate solution $u_{pod}\in V_{pod}$ to the problem \cref{eq_inverse} in $V_{pod}$ via the Galerkin projection
\begin{equation}
a(u_{pod}, \psi) = (f, \psi)_X, \quad \forall \psi \in V_{pod}.
\end{equation}
Similarly, for the problem \cref{eq_evolution}, the semi-discrete approximation $\hat{u}\in V_h$ is defined by
\begin{equation}
(\hat{u}_t, \psi)_X + a(\hat{u}, \psi) = (f, \psi)_X, \quad \forall \psi \in V_{h},
\end{equation}
where $V_h\subset X$ denotes the discrete subspace.

\begin{theorem}\label{error_Ritz}
    Let $\mathcal{P}^{N_{pod}}$ denote the Ritz projection onto $V_{pod}$ with respect to the bilinear form $a(\cdot,\cdot)$, defined by $a(\mathcal{P}^{N_{pod}}u,v)=a(u,v)$  for all $\forall v\in V_{pod}$. Then, the Ritz projection error satisfies
    \begin{equation}
        \|\hat{u}(\cdot,t_j)-\mathcal{P}^{N_{pod}}\hat{u}(\cdot,t_j)|\vert_X\leq C\frac{\mu_L}{\sqrt{\mu}_1}\sqrt{\|\mathrm{V}|\vert_2\rho},\qquad \forall j=0,\ldots,M.
    \end{equation}
\end{theorem}
\begin{proof}
    By the definition of the Ritz projection, for any $\psi\in V_{pod}$, we have 
    \begin{equation}
        a(\hat{u}(\cdot,t_j) - \mathcal{P}^{N_{pod}}\hat{u}(\cdot,t_j), \psi) = 0.
    \end{equation}
    Using the orthogonality property, for any $\psi\in V_{pod}$, it holds that
    \begin{equation}
    \begin{aligned}
        \|\hat{u}(\cdot,t_j)-\mathcal{P}^{N_{pod}}\hat{u}(\cdot,t_j)|\vert^2_{\mathcal{A}}&=a(\hat{u}(\cdot,t_j)-\mathcal{P}^{N_{pod}}\hat{u}(\cdot,t_j),\hat{u}(\cdot,t_j)-\mathcal{P}^{N_{pod}}\hat{u}(\cdot,t_j))\\[2mm]
        % &=a(u(t_j)-\mathcal{P}^{N_{pod}}u(t_j),u(t_j))\\
        &=a(\hat{u}(\cdot,t_j)-\mathcal{P}^{N_{pod}}\hat{u}(\cdot,t_j),u(\cdot,t_j)-\psi).
    \end{aligned}
    \end{equation}
    By the Cauchy-Schwarz inequality and the definition of the norm $\|\cdot\|_{\mathcal{A}}$, we have
    \begin{equation}
        \lVert \hat{u}(\cdot,t_j) - \mathcal{P}^{N_{pod}} \hat{u}(\cdot,t_j) \rVert_{\mathcal{A}}^2 \le \lVert \hat{u}(\cdot,t_j) - \mathcal{P}^{N_{pod}} \hat{u}(\cdot,t_j) \rVert_{\mathcal{A}} \cdot \lVert \hat{u}(\cdot,t_j) - \psi \rVert_{\mathcal{A}}.
    \end{equation}
    Thus, $\lVert \hat{u}(\cdot,t_j) - \mathcal{P}^{N_{pod}} \hat{u}(\cdot,t_j) \rVert_{\mathcal{A}} \le \lVert \hat{u}(\cdot,t_j) - \psi \rVert_{\mathcal{A}}$. Choosing $\psi=\propod \hat{u}(\cdot,t_j)$ and applying \cref{norm_equi}, we obtain
    \begin{equation}
        \begin{aligned}
            \lVert \hat{u}(\cdot,t_j) - \mathcal{P}^{N_{pod}} \hat{u}(\cdot,t_j) \rVert_X &\le \frac{1}{\sqrt{\mu}_1} \lVert \hat{u}(\cdot,t_j) - \mathcal{P}^{N_{pod}} \hat{u}(\cdot,t_j) \rVert_{\mathcal{A}}  \\[2mm]
            &\le \frac{1}{\sqrt{\mu}_1} \lVert \hat{u}(\cdot,t_j) - \propod \hat{u}(\cdot,t_j) \rVert_A  \\[2mm]
            &\leq \sqrt{\frac{\|\mathrm{V}\|_2}{\mu_1}} \lVert \hat{u}(\cdot,t_j) - \propod \hat{u}(\cdot,t_j) \rVert_X  \\[2mm]
            &\le \frac{C\mu_L}{\sqrt{\mu}_1} \sqrt{\|\mathrm{V}\|_2 \rho}. 
        \end{aligned}
    \end{equation}
This completes the proof.

\end{proof}

Before proceeding to the main error analysis, we state a standard assumption regarding the approximation properties of the discrete subspace $V_h$.

\begin{assumption}\label[assumption]{err_galerkin_discretization}
    Let $\hat{u}$ be the exact solution to \cref{eq_evolution} and $\hat{u}_h \in V_h$ be the corresponding Galerkin solution. There exists an approximation error $ \varepsilon_h > 0$, determined by the richness of $V_h$, such that
    \begin{equation}
        \|\hat{u} - \hat{u}_h|\vert_X \leq \varepsilon_h \|f|\vert_X,
    \end{equation}
    where $\varepsilon_h \to 0$ as $V_h$ is refined (i.e., as the subspace becomes sufficiently rich).
\end{assumption}

Now, we are in a position to state the final error estimate for the POD-Galerkin approximation.
\begin{theorem}\label{err_ture_solution}
    Let $u$ be the exact solution to problem \cref{eq_forward} and $u_{pod}\in V_{pod}$ be the corresponding POD-Galerkin approximation. Then, the error satisfies
    % Using the test space $V_{pod}$ described above, the error between the approximation $u_{pod}$ and the exact solution $u^*$ of \cref{i_p_1}:
    \begin{equation}
        \|u -u_{pod}|\vert_X\leq C(\frac{e^{-\mu_1T}}{\mu_1}+\varepsilon_h +\frac{\mu_L}{
        \sqrt\mu_1}\sqrt{\|\mathrm{V}|\vert_2\rho})\|f|\vert_X,
    \end{equation}
    where $\rho=\sum_{k=N_{pod}+1}^{2M+1}\lambda_k$ represents the POD truncation error.
\end{theorem}
\begin{proof}
    Throughout this proof, $C$ denotes a generic constant that may vary from line to line. Let $u_h\in V_h$ be the solution to \cref{eq_inverse} and $\hat{u}\in V_h$ be the semi-discrete solution to \cref{eq_evolution}. These satisfy
   \setcounter{equation}{45}  
    \begin{subequations}\label{eq_fem_solution}
        \begin{align}
                a(u_h,\psi)&=(f,\psi)_X,\quad \forall \psi\in V_h, \label{eq_first}  \\[2mm] 
                (\hat u_t,\psi)_X+a(\hat u,\psi)&=(f,\psi)_X,\quad \forall \psi\in V_h. \label{eq_second}      
        \end{align}
    \end{subequations}
     We decompose the total error as 
     \begin{equation}\label{total_error}
        \begin{aligned}
            \| u^*-u_{pod}|\vert_X
        % &=\| u^*-u(T)+u(T)-\hat u(T)+\hat u(T)-\mathcal{P}^{N_{pod}}\hat u(T)+\mathcal{P}^{N_{pod}}\hat u(T)-u_{pod}|\vert_X\\[2mm]
        \leq\|\theta_1|\vert_X+\|\theta_2|\vert_X+\|\theta_3|\vert_X+\|\theta_4|\vert_X,
        \end{aligned}
    \end{equation}
    where $\theta_1= u^*-u(T)$, $\theta_2=u(T)-\hat u(T)$, $\theta_3=\hat u(t_k)-\mathcal{P}^{N_{pod}}\hat u(t_k)$, and $\theta_4=\mathcal{P}^{N_{pod}}\hat u(T)-u_{pod}$.

    From \Cref{err_gradient_flow}, \Cref{err_galerkin_discretization}, and \Cref{error_Ritz}, we obtain
    \begin{equation}
         \|\theta_1|\vert_X \leq\frac{e^{-\mu_1 T}}{\mu_1}\|f|\vert_X,\quad \|\theta_2|\vert_X \leq \varepsilon_h \|f|\vert_X,\quad \|\theta_3|\vert_X \leq \frac{C\mu_L}{\sqrt{\mu}_1}\sqrt{\|V|\vert_2\rho}.
    \end{equation}
    It remains to estimate the term $\theta_4$. Recalling that $u_{pod}=\mathcal{P}^{N_{pod}} u_h$, we have
    \begin{equation}
         \theta_4(t)=\mathcal{P}^{N_{pod}}\hat u(\cdot,t)-u_{pod}=\mathcal{P}^{N_{pod}}(\hat u(\cdot,t)-u_h).
    \end{equation}
    Subtracting the equations in \cref{eq_fem_solution} and setting
    $\psi=\mathcal{P}^{N_{pod}}(u(\cdot,t)-u_h)$, we obtain
    \begin{equation}
    \begin{aligned}
        (\hat u_t(\cdot,t),\mathcal{P}^{N_{pod}}(\hat u(\cdot,t)-u_h))_X&=a(u_h-\hat u(\cdot,t),\mathcal{P}^{N_{pod}}(u_h-\hat u(\cdot,t))\\
        &=a(\mathcal{P}^{N_{pod}}(u_h-\hat u(\cdot,t)),\mathcal{P}^{N_{pod}}(u_h-\hat u(\cdot,t))).
    \end{aligned}
    \end{equation}
    Applying the Cauchy-Schwarz inequality and the norm equivalence \cref{norm_equi}, we have 
    \begin{equation}
    \begin{aligned}
        \|\mathcal{P}^{N_{pod}}(u_h-\hat u(\cdot,t))|\vert_X^2&\leq\frac{1}{\mu_1}\|\mathcal{P}^{N_{pod}}(u_h-\hat u(\cdot,t))|\vert_{\mathcal{A}}^2\\
        &\leq \frac{1}{\mu_1}\|\hat u_t(\cdot,t)|\vert_X\|\mathcal{P}^{N_{pod}}(u_h-\hat u(\cdot,t)|\vert_X.
    \end{aligned}
    \end{equation}
    This implies $\|\theta_4\|_X \leq \frac{1}{\mu_1} \|\hat u_t(\cdot,t)|\vert_X$.
    % \begin{equation}
    %     \|\mathcal{P}^{N_{pod}}((u_h-\hat u(t))|\vert_X\leq \frac{1}{\mu_1^2} \|\hat u_t(t)|\vert_X.
    % \end{equation}
    Analogous to the estimate of $\|u_t|\vert_X$ in \Cref{err_gradient_flow}, we have $\|\hat u_t(\cdot,T)|\vert_X\leq e^{-\mu_1 T}\|f|\vert_X$. Thus,
    \begin{equation}\label{theta_4}
        \|\theta_4|\vert_X\leq  \frac{e^{-\mu_1 T}}{\mu_1} \|f|\vert_X.
    \end{equation}
    Combining these estimates into \cref{total_error}, we obtain
    \begin{equation}
        \|u^*-u_{pod}|\vert_X\leq C(\frac{e^{-\mu_1T}}{\mu_1}+\varepsilon_h+\frac{\mu_L}{\sqrt{\mu}_1}\sqrt{\|V|\vert_2\rho})\|f|\vert_X,
    \end{equation}
    where $\rho=\sum_{k=N_{pod}+1}^{2M+1}\lambda_k$. This completes the proof.

\end{proof}

\section{Application of the Pseudo-time Data-Driven POD Framework}\label{sec_pod_application}
% In the preceding sections, we presented the framework of the gradient flow method. Next, we will specifically discuss three classes of problems that are extensively applied.
Having established the pseudo-time data-driven POD framework in \cref{sec_pod_frame}, we now demonstrate its versatility by applying it to two representative classes of problems.

\subsection{Elliptic Inverse Source Problems}\label{sec_elliptic_ip}
% Having applied the proposed method to discrete linear algebraic systems, we now transition to a continuous setting. Specifically, 
We apply the pseudo-time data-driven POD framework developed in \Cref{sec_pod_frame} to the elliptic inverse source problems. To this end, Let $X = L^2(\Omega)$ where $\Omega \subset \mathbb{R}^d$ ($d \in \{1, 2, 3\}$) is an open, bounded domain with either a $C^2$ boundary $\partial\Omega$ or a convex geometry satisfying the uniform cone condition.

Let $\mathcal{L}$ be the second-order elliptic operator defined by $\mathcal{L}u = -\nabla \cdot (a(x)\nabla u) + c(x)u$. To ensure well-posedness, we impose standard assumptions on the coefficients: (i) $a \in C(\bar{\Omega})$ satisfies the uniform ellipticity condition $0 < a_1 \le a(x) \le a_2$; and (ii) $c \in C(\bar{\Omega})$ with $c(x) \ge 0$. Under these conditions, the operator $\mathcal{L}$ naturally induces an energy inner product that is topologically equivalent to the standard $H^1$ inner product, allowing us to simply define $(u, v)_{\mathcal{L}} := (u, v)_{H^1}$ for $u,v\in D(\mathcal{L})$.

By setting the domain as $D(\mathcal{L}) = H^2(\Omega) \cap H^1_0(\Omega)$, the Lax-Milgram lemma and classical elliptic regularity theory guarantee that $\mathcal{L}$ is densely defined, self-adjoint, and strictly positive-definite on $L^2(\Omega)$. Consequently, for any source term $f \in L^2(\Omega)$, the corresponding homogeneous Dirichlet problem
\begin{equation}\label{elliptic_equation}
    \begin{aligned}
        \left\{\begin{array}{ll}
             \mathcal{L} u = f(x) &\text { in }~ \Omega, \\[2mm] 
            u(x)= 0 &\text { on }~ \partial \Omega.
            % u(x,0)= 0 &\text { in }~ \Omega.
        \end{array}\right.
    \end{aligned}  
\end{equation}
admits a unique solution $u \in D(\mathcal{L})$. The forward solution operator $\mathcal{L}^{-1}$ is therefore linear, compact, and self-adjoint, precisely realizing the abstract static model $u = \mathcal{L}^{-1}f$.

However, as emphasized in \Cref{sec_pseudo_time}, a single stationary state is insufficient to provide the rich dataset required for extracting the POD basis functions. To generate the necessary data snapshots, we apply the proposed pseudo-time scheme to this setting. By embedding the static problem \eqref{elliptic_equation} into a dynamic framework with a pseudo-time variable $t$, we construct the following parabolic initial-boundary value problem:
\begin{equation}\label{parabolic_equation}
    \begin{aligned}
        \left\{\begin{array}{lll}
             \hat u_t + \mathcal{L} \hat u = f(x) &\text { in }~ \Omega \times (0, T), \\[2mm] 
            \hat u(x,t)= 0 &\text { on }~ \partial \Omega \times (0, T),\\[2mm]
            \hat u(x,0)= 0 &\text { in }~ \Omega.\end{array}\right.
        \end{aligned}
\end{equation}
This system serves as the exact realization of the abstract evolution equation \cref{eq_evolution}, where the operator $\mathcal{L}$ and the unknown source $f$ remain identical to those in the static case.

To rigorously verify that this setup satisfies \Cref{compact_self_adjoint}, we must demonstrate that the elliptic operator $\mathcal{L}$ and its solution operator $\mathcal{L}^{-1}$ possess the required spectral properties. In our current setting, $\mathcal{L}$ naturally assumes the role of the abstract unbounded operator $\mathcal{A}$, while $\mathcal{L}^{-1}$ corresponds to the compact operator $\mathcal{R}$. The following lemma recalls the classical spectral theory for second-order elliptic operators \cite{agmon2010lectures,fleckinger1986eigenvalues}, confirming that $\mathcal{L}$ provides the exact discrete spectrum and orthonormal basis demanded by our abstract framework.

\begin{lemma}\label[lemma]{elliptic_op_eigenvalue}
    Suppose $\Omega$ is a bounded domain in $\bbR^d$,  and the coefficients satisfy $a$, $c \in C^0(\bar{\Omega})$, $c(x)\geq 0$. Then, the eigenvalue problem
    \begin{equation}\label{elliptic_op_eige_problem}
        \caL \phi = \mu \phi, \quad \text{with} \quad \phi_{\partial \Omega} =0, 
    \end{equation}
    has a countable set of positive eigenvalues $0 < \mu_1 \leq \mu_2 \leq \cdots$. The corresponding eigenfunctions $\{ \phi_k\}_{k=1}^{\infty}$ form a complete orthonormal basis for $L^2(\Omega)$. Moreover, there exist positive constants $C_1, C_2>0$ such that $C_1 k^{\sfrac{2}{d}}\leq \mu_k \leq C_2 k^{\sfrac{2}{d}}$ for all $k=1, 2,  \cdots$.
\end{lemma}

By \Cref{elliptic_op_eigenvalue}, the following theorem provides the exponential error bounds between the dynamic trajectory $\hat{u}(\cdot,t)$ of \cref{parabolic_equation} and the stationary solution $u$ of \cref{elliptic_equation}.

\begin{theorem}
    Let $u$ be the solution to the problem \cref{elliptic_equation} and $\hat{u}(\cdot,t)$ be the solution to the problem \cref{parabolic_equation} for a given $f\in L^2(\Omega)$. Then, for any $T>0$, the following error estimates hold:
    \begin{equation}
        \| \hat{u}(\cdot,T) - u \|_{L^2(\Omega)} \leq \frac{e^{-\mu_1 T}}{\mu_1} \| f \|_{L^2(\Omega)} \quad \text{and}\quad \|\hat{u}_t(\cdot,T)|\vert_X \leq e^{-\mu_1 T}\|f|\vert_{L^2(\Omega)}.
    \end{equation}
\end{theorem}

As discussed in \Cref{sec_pod_construct}, the lack of explicit knowledge of $f$ renders the forward trajectory $\hat{u}(\cdot,t)$ of \cref{parabolic_equation} computationally inaccessible for snapshot generation. To bypass this, we adapt the abstract data-driven adjoint method to our specific PDE setting. By employing the observed data $m(x)$ as a surrogate source, we generate the requisite snapshots through the following adjoint equation to construct the POD subspace $V_{pod}$:

\begin{equation}\label{adjoint_parabolic_equation}
    \begin{aligned}
        \left\{\begin{array}{lll}
             \widetilde u_t + \mathcal{L} \widetilde u = m(x) &\text { in }~ \Omega \times (0, T), \\[2mm] 
            \widetilde u(x,t)= 0 &\text { on }~ \partial \Omega \times (0, T),\\[2mm]
            \widetilde u(x,0)= 0 &\text { in }~ \Omega.\end{array}\right.
        \end{aligned}
\end{equation}
It is crucial to emphasize that, consistent with the theoretical justification provided in \Cref{sec_conv_pod}, the unknown source $f$ must be assumed to reside within an $L$-dimensional subspace (with $L \leq M$). This structural assumption guarantees that the accessible snapshot space generated by the adjoint equation \cref{adjoint_parabolic_equation} perfectly coincides with the inaccessible snapshot space of the forward equation \cref{parabolic_equation}.

For the parabolic equation \cref{parabolic_equation}, standard linear finite element analysis \cite[Theorem 1.1]{thomee1990finite} yields an $L^2$-error of $\mathcal{O}(h^2)$, which directly verifies \Cref{err_galerkin_discretization} with $\varepsilon_h = \mathcal{O}(h^2)$.
Based on this, we now present the main convergence result for our proposed POD reduced-order method.

% While the Galerkin discretization of the classical parabolic equation \cref{parabolic_equation} has been extensively studied with various outcomes, we focus here on a result derived from the classical finite element method.
% \begin{lemma}[Theorem 1.1, \cite{thomee1990finite}]
%     Let $V_h \subset H_0^1(\Omega)$ be the space of piecewise linear continuous finite elements in a quasi-uniform triangulation of $\Omega$ with mesh size $h>0$.  Define $\hat u(t)$ is the exact solution of \cref{parabolic_equation}, and $\hat u_h(t) \in V_h$ is the Galerkin solution. Assume that $f \in L^2(0,T; L^2(\Omega))$. For every $t\in[0,T]$,
%     \begin{equation}
%         \|\hat u(t)-\hat u_h(t)|\vert_{L^2}\leq Ch^2\|f|\vert_{L^2},
%     \end{equation}
%     where the constant $C>0$ is independent of  of $h$ and $f$.
% \end{lemma}

% In the same way, we clarify the definitions of the new inner product $(u,v)_{\mathcal{L}}=(u,v)_{H^1}$ and the coefficient matrix $V_{ij}=(\psi_i,\psi_j)_{H^1}$.

\begin{theorem}
    Let $u$ be the solution to problem \cref{elliptic_equation} and $u_{pod}\in V_{pod}$ be the corresponding POD-Galerkin approximation. Then, the error satisfies
    % Using the test space $V_{pod}$ described above, we obtain the Galerkin solution $u_{pod}$. Consequently, the approximation error between the exact solution  of elliptic equation \cref{elliptic_equation} $u^*$ and $u_{pod}$ is defined as:
    \begin{equation}
        \|u-u_{pod}|\vert_{L^2}\leq C(e^{-C_1 T}+h^2+L^{\frac{2}{d}}\sqrt{\|V|\vert_2\rho})\|f|\vert_{L^2},
    \end{equation}
    where $\rho=\sum_{k=N_{pod}+1}^{2M+1}\lambda_k$ represents the POD truncation error.
    % and the positive constant $C$ is independent of $h$ and $f$.
\end{theorem}
\begin{proof}
    It suffices to observe that $\mu_1 \geq C_1$ and $\mu_L\leq C_2 L^{\frac{2}{d}}$. This completes the proof.
    % By substituting the above two lemmas into \Cref{err_ture_solution}, we obtain the desired result.
\end{proof}

\subsection{Fredholm Integral Equations of the First Kind}\label{sec_FIEFK} 

Let $\Omega\subset\mathbb R^d$ ($d \in \{1, 2, 3\}$) be a bounded Lipschitz domain.
We consider the Fredholm integral equation of the first kind:
\begin{equation}\label{eq_int_op_eq}
    \int_{\Omega} K(y, x) f(x) \rmd x = u(y),
\end{equation}
where $f, u \in L^2(\Omega)$ represent the unknown variable and the observation variable, respectively.
% and $K(y, x)$ denotes the Point Spread Function (PSF). 
This equation defines an integral operator $\mathcal{K} : L^2(\Omega) \to L^2(\Omega)$ given by 
\[
(\mathcal{K}f)(y) = \int_{\Omega} K(y, x) f(x) \rmd x.
\]
% $(\mathcal{K}f)(y) = \int_{\Omega} K(y, x) f(x) \rmd x$. 

Furthermore, assume that \(K(y,x)\in L^2(\Omega\times\Omega)\), \(K(y,x)=K(x,y)\), and \(\ker(\mathcal K)=\{0\}\). Then \(\mathcal K\) is a compact self-adjoint operator. Moreover, since
\[
\overline{\operatorname{Ran}(\mathcal K)}
=
\ker(\mathcal K^*)^\perp
=
\ker(\mathcal K)^\perp
=
L^2(\Omega),
\]
the range \(\operatorname{Ran}(\mathcal K)\) is dense in \(L^2(\Omega)\). Consequently, the inverse operator
\[
\mathcal K^{-1}:\operatorname{Ran}(\mathcal K)\to L^2(\Omega)
\]
is well defined on the dense domain \(\operatorname{Ran}(\mathcal K)\). Since \(\mathcal K\) is compact and injective on an infinite-dimensional space, \(\mathcal K^{-1}\) is generally unbounded.

Subsequently, we approximate \(\mathcal K\) by a
quadrature-based finite-rank operator. The Nystr\"om discretization of Fredholm
integral operators is classical; see, e.g.,~\cite{kress1989linear} for convergence results in
spaces of continuous functions. Combining this with the Sobolev embedding theorem, we obtain the following error estimate.

\begin{theorem}\label{thm_inv_dis_err}
Let \(\Omega\subset\mathbb{R}^d\) be a bounded Lipschitz domain, and define
\begin{equation}
    (\mathcal K_h f_h)(y_i)=\sum_{j=1}^n w_jK(y_i,x_j)f(x_j),
\end{equation}
where \(\{x_j\}_{j=1}^n\subset\Omega\) are the quadrature nodes, \(\{w_j\}_{j=1}^n\) are the corresponding weights, and $\{y_i\}_{i=1}^n \subset \Omega$ is a set of evaluation points.
Assume that \(m>\frac{d}{2}\) and that the quadrature rule satisfies
\begin{equation}
    \left|
\int_\Omega v(x)\,dx-\sum_{j=1}^n w_jv(x_j)
\right|
\le Ch^p\|v\|_{H^m(\Omega)},
\qquad \forall v\in H^m(\Omega).
\end{equation}
Assume moreover that
\[
\sup_{y\in\Omega}\|K(y,\cdot)\|_{W^{m,\infty}(\Omega)}<\infty.
\]
Then, for every \(f\in H^m(\Omega)\),
\begin{equation}
    \|\mathcal K f-\mathcal K_h f_h \|_{\infty}
\le Ch^p\|f\|_{H^m(\Omega)}.
\end{equation}
\end{theorem}
% \begin{proof}
% Since \(m>1\) and \(\Omega\subset\mathbb{R}^2\), the Sobolev embedding theorem implies \(H^m(\Omega)\hookrightarrow C^0(\overline{\Omega})\), so \(f(x_j)\) is well defined. For fixed \(y\in\Omega\), let \(v_y(x)=K(y,x)f(x)\). Then
% \[
% (\mathcal K f)(y)-(\mathcal K_h f)(y)
% =
% \int_\Omega v_y(x)\,dx-\sum_{j=1}^N w_jv_y(x_j).
% \]
% By the quadrature estimate,
% \[
% |(\mathcal K f)(y)-(\mathcal K_h f)(y)|
% \le Ch^p\|v_y\|_{H^m(\Omega)}.
% \]
% The product estimate gives
% \[
% \|v_y\|_{H^m(\Omega)}
% \le C\|K(y,\cdot)\|_{W^{m,\infty}(\Omega)}\|f\|_{H^m(\Omega)}
% \le C\|f\|_{H^m(\Omega)}.
% \]
% Hence
% \[
% |(\mathcal K f)(y)-(\mathcal K_h f)(y)|
% \le Ch^p\|f\|_{H^m(\Omega)}.
% \]
% Taking the \(L^2(\Omega)\)-norm in \(y\) yields the result.
% \end{proof}
Hence, based on the same quadrature nodes, the integral equation \cref{eq_int_op_eq} can be discretized into the linear algebraic system
\begin{equation}\label{eq_int_dis}
    \mathcal{K}_h f_h = u_h,
\end{equation}
where $f_h,u_h\in \bbR^n$, and $\mathcal{K}_h\in \bbR^{n\times n}$ is a symmetric positive definite (SPD) matrix.

Next, to generate the snapshots required for constructing the POD subspace, we introduce an artificial continuous time variable $t$ and formulate the following evolution equation:
\begin{equation}\label{eq_int_op_evlo}
    \begin{aligned}
        \left\{\begin{array}{ll}
             (\hat u_{h})_t + \mathcal{K}_h^{-1} \hat u_h = f_h, & t>0, \\[2mm] 
            % \hat u (x,t)= 0 &\text { on }~ \partial \Omega \times (0, T),\\[2mm]
            \hat u_h(0)= 0 .\end{array}\right.
        \end{aligned}
\end{equation}

% Since $\mathcal{K}_h$ is a compact self-adjoint operator, the spectral theorem guarantees that its inverse $\mathcal{K}_h^{-1}$ possesses a discrete spectrum strictly bounded from below. Consequently, \Cref{compact_self_adjoint} is satisfied, which ensures the exponential convergence of the dynamic system and leads to the following error estimates.

%According to the analysis in \Cref{sec_LAE}, \Cref{compact_self_adjoint} is satisfied. 
It is well established that SPD matrices satisfy the following diagonalization property, which ensures that $\mathcal{K}_h^{-1}$ fulfills \Cref{compact_self_adjoint}.
\begin{lemma}\label[lemma]{Linear_op_eigenvalue}
    Let  $\mathcal{K}_h\in \bbR^{n\times n}$ be a SPD matrix. Then, $\mathcal{K}_h$ can be orthogonally diagonalized as follows:
    \begin{equation}
        \mathcal{K}_h Q=Q \Lambda, \quad \Lambda =\operatorname{diag}\{1/\mu_1,\cdots,1/\mu_n\},
    \end{equation}
    where $1/\mu_1\geq \cdots\geq 1/\mu_n >0$ are the eigenvalues of $\mathcal{K}_h$, and the columns of $Q=[\phi_1,\ldots,\phi_n]$ are the corresponding orthonormal eigenvectors.
    Consequently, the inverse matrix $\mathcal{K}_h^{-1}$ possesses the same set of eigenvectors $\{\phi_i\}_{i=1}^n$ associated with the eigenvalues $0<\mu_1\leq\cdots\leq \mu_n$.
\end{lemma}

Since \Cref{Linear_op_eigenvalue} verifies that the inverse matrix $\mathcal{K}_h^{-1}$ fulfills \Cref{compact_self_adjoint}, we can directly apply the general convergence result from \Cref{err_gradient_flow} to this finite-dimensional case. Therefore, we have the following error estimates:
\begin{theorem}
    Let $u_h$ be the solution to the problem \cref{eq_int_dis} and $\hat{u}_h(t)$ be the solution to the problem \cref{eq_int_op_evlo} for a given $f_h\in \bbR^n$. Then, for any $T>0$, the following error estimates hold:
    \begin{equation}
        \| \hat{u}_h(T) - u_h \|_{2} \leq \frac{e^{-\mu_1 T}}{\mu_1} \| f_h \|_{2}\quad \text{and}\quad \|(\hat{u}_{h})_t(T)|\vert_{2} \leq e^{-\mu_1 T}\|f_h|\vert_{2},
    \end{equation}
    where $\mu_1 > 0$ is the smallest eigenvalue of $\mathcal{K}_h^{-1}$.
\end{theorem}

Since $f_h$ is the unknown, the forward trajectory \cref{eq_int_op_evlo} is computationally inaccessible. To generate snapshots, we replace $f_h$ with the observed data $m_h$ and formulate the following computable adjoint equation:

\begin{equation}\label{eq_int_pod_basis_generate}
    \begin{aligned}
        \left\{\begin{array}{ll}
             (\widetilde u_{h})_t + \mathcal{K}_h^{-1} \widetilde u_h = m_h, &t > 0, \\[2mm] 
             % \widetilde u (x,t)= 0 &\text { on }~ \partial \Omega \times (0, T),\\[2mm]
            \widetilde u_h(0)= 0.\end{array}\right.
        \end{aligned}
\end{equation}

\begin{remark}
    To solve \cref{eq_int_pod_basis_generate}, it is not necessary to compute the inverse matrix $\mathcal{K}_h^{-1}$ explicitly. Multiplying both sides of \cref{eq_int_pod_basis_generate} by $\mathcal{K}_h$, we obtain the equivalent formulation
    \begin{equation}
        \mathcal{K}_h (\widetilde u_{h})_t + \widetilde u_h = \mathcal{K}_h m_h.
    \end{equation}
\end{remark}

Once the POD subspace $V_{pod}$ is extracted from these adjoint snapshots, we compute the reduced-order approximation $u_{pod} \in V_{pod}$. The total error of this approximation is rigorously bounded as follows.

\begin{theorem}
    Let $u$ be the solution to problem \cref{eq_int_op_eq} and $u_{pod}\in V_{pod}$ be the corresponding POD-Galerkin approximation. Then, the error satisfies
    \begin{equation}
        \|u-u_{pod}|\vert_{n}\leq C(\frac{e^{-\mu_1 T}}{\mu_1}+h^p+\frac{\mu_L}{\sqrt{\mu_1}}\sqrt{\|V|\vert_2\rho})\|f|\vert_{H^m},
    \end{equation}
    where $\rho=\sum_{k=N_{pod}+1}^{2M+1}\lambda_k$, and $\|\cdot \|_n$ is defined by $\|u|\vert_n=(\sum_{i=1}^n (u_i)^2/n)^{1/2}$, $\|v|\vert_n=(\sum_{i=1}^n(v(x_i)^2/n)^{1/2}$, $\forall u\in \bbR^n, v\in C(\bar \Omega)$.
    % and the positive constant $C$ is independent of $h$ and $f$.
\end{theorem}
\begin{proof}
    Because the inverse matrix $\mathcal{K}_h^{-1}$ is inherently SPD, the general norm equivalence established in \cref{norm_equi} naturally reduces to the discrete case. By equipping $\bbR^n$ with the energy inner product $(\mathbf{z}_1, \mathbf{z}_2)_{\mathcal{K}_h^{-1}} := (\mathcal{K}_h^{-1}\mathbf{z}_1, \mathbf{z}_2)_2$ for $\mathbf{z}_1, \mathbf{z}_2 \in \bbR^n$, we obtain the following explicit bounds for any $\mathbf{z} \in \bbR^n$:
    \begin{equation}
        \sqrt{\mu}_1 \|\mathbf{z}\|_2 \le \|\mathbf{z}\|_{\mathcal{K}_h^{-1}} \le \sqrt{\|\mathrm{V}\|_2} \|\mathbf{z}\|_2,
    \end{equation}
    where $\mathrm{V} \in \bbR^{N\times N}$ ($N$ is the the dimension of the snapshot space) is the Gram matrix with $\mathrm{V}_{ij}=(\mathbf{\psi}_i,\mathbf{\psi}_j)_{\mathcal{K}_h^{-1}}$, and $\mathbf{\psi}_i$ represents the POD basis vectors. 

    Finally, by setting $\varepsilon_h =0$ in \Cref{err_galerkin_discretization}, a direct application of \Cref{err_ture_solution} yields the error estimate for the proposed method applied to the integral equation:
    %Based on the analysis of the matrix in \Cref{sec_LAE}, we have
    \begin{equation}
        \|u_h-u_{pod}\|_2\leq C(\frac{e^{-\mu_1 T}}{\mu_1}+\frac{\mu_L}{\sqrt{\mu_1}}\sqrt{\|V|\vert_2\rho})\|f_h|\vert_2.
    \end{equation}
    Combining the Sobolev embedding theorem, we get
    \begin{equation}
        \frac{1}{\sqrt{n}}\|f_h|\vert_2\leq \|f_h|\vert_{\infty}\leq \|f|\vert_{H^m(\Omega)}.
    \end{equation}
    Using \Cref{thm_inv_dis_err} and the triangle inequality, we obtain
    \begin{equation}
        \|u-u_{pod}|\vert_{n}\leq C(\frac{e^{-\mu_1 T}}{\mu_1}+h^p+\frac{\mu_L}{\sqrt{\mu_1}}\sqrt{\|V|\vert_2\rho})\|f|\vert_{H^m}.
    \end{equation}
    This completes the proof.
\end{proof}

\section{Numerical Experiments}\label{sec_pod_num}
In this section, we present a comprehensive series of numerical experiments to 
% validate the theoretical error estimates established in \Cref{sec_pod_application} and to 
demonstrate the computational efficiency of the proposed method. The corresponding optimization problem is formulated as
\begin{equation}\label{opti_problem}
    \min_{f\in X} \| \mathcal{R} f -m \|_n^2 + \lambda \|f\|_X^2,
\end{equation}
where $m$ denotes the observed data and $\lambda>0$ is the regularization parameter. Problem
\Cref{opti_problem} is solved by the gradient descent method.
All computations are implemented in MATLAB and executed on a 64-bit workstation equipped with an Intel Core Ultra 9 285H processor (2.90 GHz), 64 GB of RAM, and an NVIDIA RTX PRO 2000 GPU.

\subsection{Elliptic inverse source problem}

We consider the inverse source problem on the unit square $\Omega=(0,1)^2$, governed by the following equation:
\begin{equation}\label{num_exa_inv_pro_eq}
    \begin{aligned}
        \left\{\begin{array}{ll}
              \mathcal{L} u = f(\boldsymbol{x}) &\text { in }~ \Omega, \\[2mm] 
            u(\boldsymbol{x})= 0 &\text { on }~ \partial \Omega，\\[2mm]
            % u(x,0)= 0 &\text { in }~ \Omega.
        \end{array}\right.
    \end{aligned} 
\end{equation}
where $\mathcal{L}$ is defined as $\mathcal{L}u = -\nabla \cdot (a(\boldsymbol{x})\nabla u) + c(\boldsymbol{x})u$. For $\boldsymbol{x} = (x,y) \in \Omega$, the spatially varying coefficients are given by $a(\boldsymbol{x}) = 1 + 0.3\sin(\pi x)\sin(\pi y)$ and $c(\boldsymbol{x}) = \sin(\pi x)\sin(\pi y)$.

The objective is to reconstruct the unknown source $f(\boldsymbol{x})$ from the observed data $m(\boldsymbol{x})$. To generate the data, we solve \eqref{num_exa_inv_pro_eq} with the exact source using continuous piecewise linear ($\mathbb{P}_1$) finite elements on a uniform spatial mesh with mesh size $h = 1/50$. The resulting finite element solution is then interpolated onto a $300 \times 300$ grid of uniformly distributed sensors. 

To construct the POD basis, $500$ snapshots are extracted over the time interval $(0, 1]$ by solving the adjoint equation \eqref{adjoint_parabolic_equation}. This time evolution is discretized on the same spatial mesh using a backward Euler scheme with a time step size of $\Delta t = 1/500$. Moreover, For the optimization problem \cref{opti_problem}, the penalty norm is uniformly defined as $\|\cdot\|^2_X = \|\cdot \|_{L^2(\Omega)}^2$, and the gradient descent algorithm is terminated once 
% the relative $L^2$ error drops below the $10^{-5}$ threshold. 
the error satisfies $\| f_{fem} - f^*\|_{L^2(\Omega)}^2 \leq \varepsilon$ (or $\| f_{pod} - f^*\|_{L^2(\Omega)}^2 \leq \varepsilon$) for a prescribed tolerance  $\varepsilon$.

\begin{myexample}\label{elliptic_pod_performance}
    We validate the proposed method against the FEM using $9$ POD basis functions under two source configurations.
    For the double-Gaussian source, given by
    \begin{equation*}
        f = \exp\left(-\frac{(x-0.3)^2+(y-0.3)^2}{0.02}\right) + \exp\left(-\frac{(x-0.7)^2+(y-0.7)^2}{0.02}\right),
    \end{equation*}
    we set $\lambda = 6.1054 \times 10^{-10}$ and $\varepsilon = 1.0 \times 10^{-5}$. The reconstructed results and computational costs are shown in \Cref{gaussian_compar_effectiveness_pod_fem} and \Cref{Tab_gaussian_fem_pod}.
    % For the star-shaped source, we set $\lambda = 6.1136 \times 10^{-10}$ and $\varepsilon = 1.0 \times 10^{-2}$. The corresponding reconstruction and time comparisons are illustrated in \Cref{star_compar_effectiveness_pod_fem} and \Cref{Tab_star_fem_pod}.

     % In this example, the validity of the proposed method is demonstrated through a comparison with the FEM, where 9 POD basis functions are employed. The exact source term is given by $f^* = e^{-\frac{(x-0.3)^2+(y-0.3)^2}{0.02}} + e^{-\frac{(x-0.7)^2+(y-0.7)^2}{0.02}}$. For the optimization problem \cref{opti_problem}, we set the regularization parameter to $\lambda = 6.1054 \times 10^{-10}$ and define the penalty norm as $\|\cdot\|^2_X = \|\cdot \|_{L^2(\Omega)}^2$. The gradient descent iteration is terminated when the relative $L^2$ error falls below a threshold of $10^{-5}$. The corresponding reconstruction results and computational times are presented in \Cref{gaussian_compar_effectiveness_pod_fem} and \Cref{Tab_gaussian_fem_pod}, respectively.

     % The exact source term is $f$ of star-shaped function
     % For the optimization problem \cref{opti_problem}, we set the regularization parameter to $\lambda = 6.1136 \times 10^{-10}$ and define the penalty norm as $\|\cdot\|^2_X = \|\cdot \|_{L^2(\Omega)}^2$. The gradient descent iteration is terminated when the relative $L^2$ error falls below a threshold of $10^{-5}$. The corresponding reconstruction results and computational times are presented in \Cref{star_compar_effectiveness_pod_fem} and \Cref{Tab_star_fem_pod}, respectively.

    \begin{figure}[htbp]
    \centering
    % 第一个子图
    \begin{subfigure}[t]{0.3\textwidth} 
        \vspace{0pt} 
        \centering
        \includegraphics[width=\textwidth]{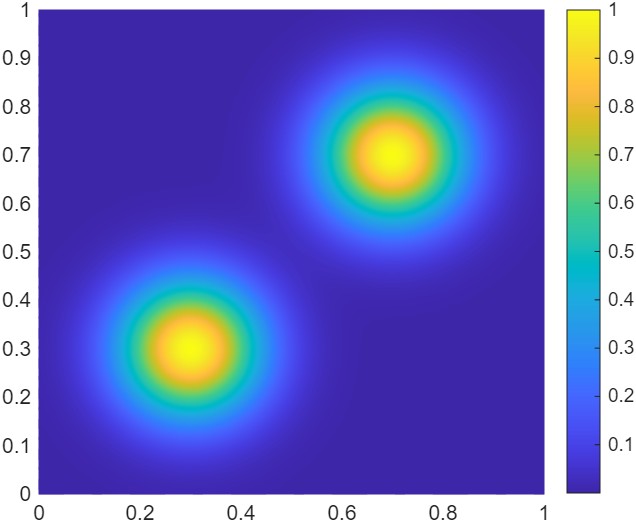}
        \caption{Exact source term} 
        \label{gaussian_exact} 
    \end{subfigure}
    \hfill % 在两个子图之间添加水平间距，使其左右对齐
    % 第二个子图
    \begin{subfigure}[t]{0.3\textwidth}  % 同上，保证两个子图加起来不超过textwidth并留有间隙
    \vspace{0pt} 
        \centering
        \includegraphics[width=\textwidth]{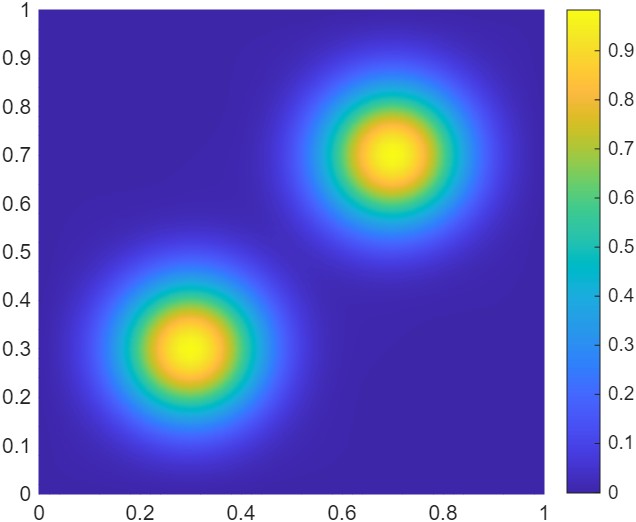}
        \caption{Reconstruction result by the FEM} % 子图的独立标题
        \label{gaussian_fem} % 子图的独立标签
    \end{subfigure}
    \hfill % 在两个子图之间添加水平间距，使其左右对齐
    % 第三个子图
    \begin{subfigure}[t]{0.3\textwidth}  % 同上，保证两个子图加起来不超过textwidth并留有间隙
    \vspace{0pt} 
        \centering
        \includegraphics[width=\textwidth]{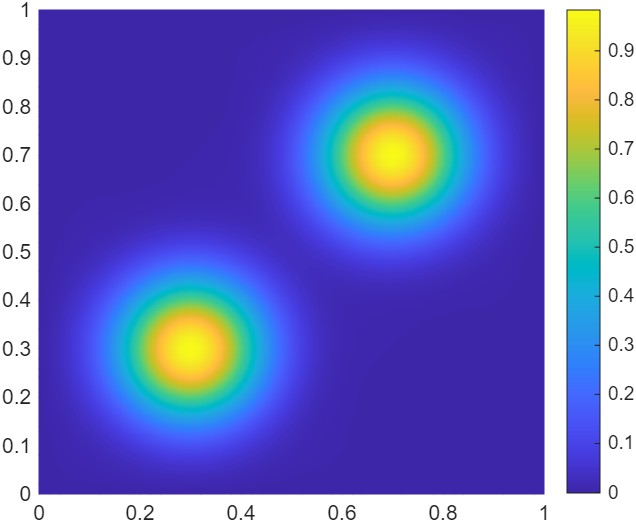}
        \caption{Reconstruction result by the POD} % 子图的独立标题
        \label{gaussian_pdo} % 子图的独立标签
    \end{subfigure}
    % 整个图的共同标题
    \caption{ Comparison of FEM and POD reconstruction results for  double-Gaussian source in \Cref{elliptic_pod_performance}}
    \label{gaussian_compar_effectiveness_pod_fem} % 整个图的共同标签
    \end{figure}

    % \begin{figure}[htbp]
    % \centering
    % % 第一个子图
    % \begin{subfigure}[t]{0.3\textwidth} 
    %     \vspace{0pt} 
    %     \centering
    %     \includegraphics[width=\textwidth]{pictures/star.jpg}
    %     \caption{Exact source term} 
    %     \label{star_exact} 
    % \end{subfigure}
    % \hfill % 在两个子图之间添加水平间距，使其左右对齐
    % % 第二个子图
    % \begin{subfigure}[t]{0.3\textwidth}  % 同上，保证两个子图加起来不超过textwidth并留有间隙
    % \vspace{0pt} 
    %     \centering
    %     \includegraphics[width=\textwidth]{pictures/star_fem.jpg}
    %     \caption{Recovered result by the FEM} % 子图的独立标题
    %     \label{star_fem} % 子图的独立标签
    % \end{subfigure}
    % \hfill % 在两个子图之间添加水平间距，使其左右对齐
    % % 第三个子图
    % \begin{subfigure}[t]{0.3\textwidth}  % 同上，保证两个子图加起来不超过textwidth并留有间隙
    % \vspace{0pt} 
    %     \centering
    %     \includegraphics[width=\textwidth]{pictures/star_pod.jpg}
    %     \caption{Recovered result by the POD} % 子图的独立标题
    %     \label{star_pod} % 子图的独立标签
    % \end{subfigure}
    % % 整个图的共同标题
    % \caption{Comparison of FEM and POD reconstruction results for $f^*$ of star-shaped source in \Cref{elliptic_pod_performance}}
    % \label{star_compar_effectiveness_pod_fem} % 整个图的共同标签
    % \end{figure}

    \begin{table}[ht]
    \caption{ Comparison of computational time between FEM and POD for double-Gaussian source in \Cref{elliptic_pod_performance}}
    \centering
    % \footnotesize
    % \normalsize 
    % \scalebox{1.2}{
    \begin{tabular}{lccc} 
    \toprule % 加粗顶部横线（1.5pt 线宽）
       Method    & Off. (s)                    & Opti. (s)     & Total (s)       \\ \midrule
      FEM                          &-- & 16.042 & 16.042  \\ \midrule
      POD                        & 1.334 &3.079   & 4.413 \\ \bottomrule
    \end{tabular}
    \label{Tab_gaussian_fem_pod}
    \end{table}

    % \begin{table}[ht]
    % \caption{ Comparison of computational time between FEM and POD for $f^*$ of star-shaped source in \Cref{elliptic_pod_performance}}
    % \centering
    % % \footnotesize
    % % \normalsize 
    % % \scalebox{1.2}{
    % \begin{tabular}{lccc} 
    % \toprule % 加粗顶部横线（1.5pt 线宽）
    %    Method     & Off. (s)                    & Opti. (s)     & Total (s)       \\ \midrule
    %   FEM                          &-- & 59.239 & 59.239  \\ \midrule
    %   POD                        & 3.493 &10.394   & 13.887 \\ \bottomrule
    % \end{tabular}
    % \label{Tab_star_fem_pod}
    % \end{table}
\end{myexample}

\begin{myexample}\label{example_elliptic_pod_performance}
    To evaluate the algorithm's robustness against varying noise levels, we examine a star-shaped source. For a $2$\% noise level, the reconstruction utilizes $7$ POD basis functions with a fixed tolerance of $\varepsilon = 1.31 \times 10^{-2}$ and a regularization parameter of  $\lambda = 1.0153 \times 10^{-7}$. When the noise level increases to $5$\%, we employ $10$ POD basis functions, setting the tolerance to $\varepsilon = 1.92 \times 10^{-2}$ and $\lambda = 3.6326 \times 10^{-7}$. As illustrated in \Cref{star_compar_effectiveness_pod_fem}, the reconstructed results successfully preserve the structural integrity of the source despite these perturbations.  The corresponding computational costs are shown in \Cref{Tab_star_fem_pod}.
    
    % a multi-peak Gaussian source $f$ given by
    % \begin{equation}\label{mutl_peak_Gaussian_eq}
    % \begin{aligned}
    %     f =\ & 0.8\exp\left(-\frac{(x-0.3)^2+(y-0.3)^2}{0.02}\right) + 0.6\exp\left(-\frac{(x-0.8)^2+(y-0.6)^2}{0.015}\right) \\
    %      & + 0.4\exp\left(-\frac{(x-0.5)^2+(y-0.8)^2}{0.01}\right) + 0.5\exp\left(-\frac{(x-0.2)^2+(y-0.7)^2}{0.008}\right).
    % \end{aligned}
    % \end{equation}
    % Using 7 POD basis functions, we perform the reconstructions with a fixed tolerance of $\varepsilon = 1.31 \times 10^{-2}$. The regularization parameter is set to 
    % $\lambda = 1.0153 \times 10^{-7}$ for a 2\% noise level.
    % Using 10 POD basis functions, we perform the reconstructions with a fixed tolerance of $\varepsilon = 1.92 \times 10^{-2}$. The regularization parameter is set to 
    % $\lambda = 3.6326 \times 10^{-7}$ for a 5\% noise level.
    % % and increased to $\lambda = 1.2014 \times 10^{-7}$ for a 3\% noise level. 
    % As depicted in \Cref{star_compar_effectiveness_pod_fem}, the reconstructed results successfully preserve the structural integrity of the source despite these perturbations.

     \begin{figure}[htbp]
    \centering
    % 第一个子图
    \begin{subfigure}[t]{0.3\textwidth} 
        \vspace{0pt} 
        \centering
        \includegraphics[width=\textwidth]{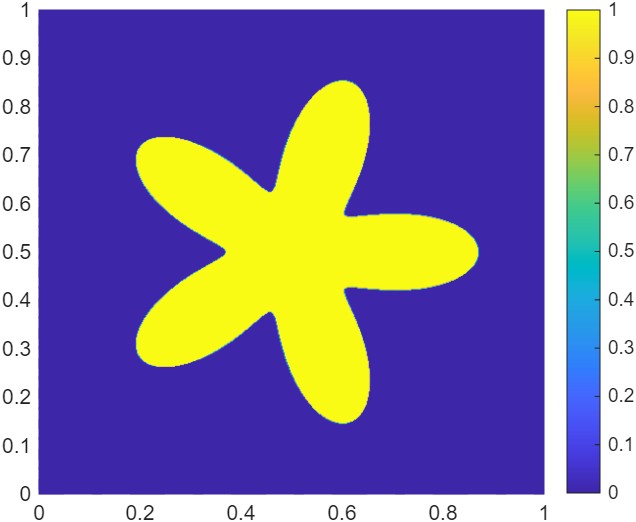}
        \caption{Exact source term} 
        \label{star_exact} 
    \end{subfigure}
    \hfill % 在两个子图之间添加水平间距，使其左右对齐
    % 第二个子图
    \begin{subfigure}[t]{0.3\textwidth}  % 同上，保证两个子图加起来不超过textwidth并留有间隙
    \vspace{0pt} 
        \centering
        \includegraphics[width=\textwidth]{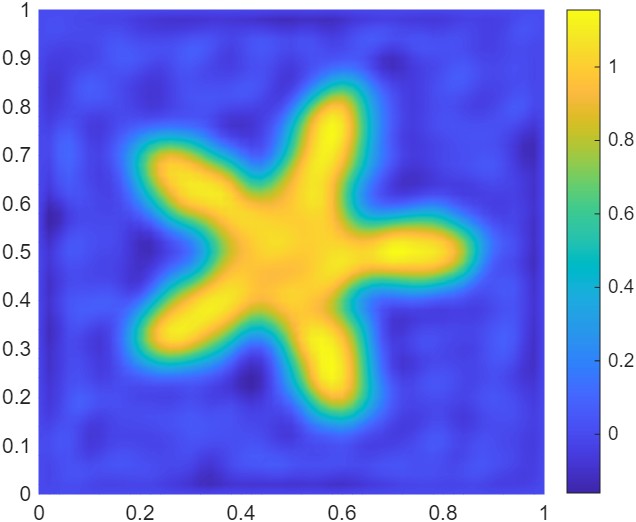}
        \caption{Recovered result by the FEM with 2\% noise} % 子图的独立标题
        \label{star_fem} % 子图的独立标签
    \end{subfigure}
    \hfill % 在两个子图之间添加水平间距，使其左右对齐
    % 第三个子图
    \begin{subfigure}[t]{0.3\textwidth}  % 同上，保证两个子图加起来不超过textwidth并留有间隙
    \vspace{0pt} 
        \centering
        \includegraphics[width=\textwidth]{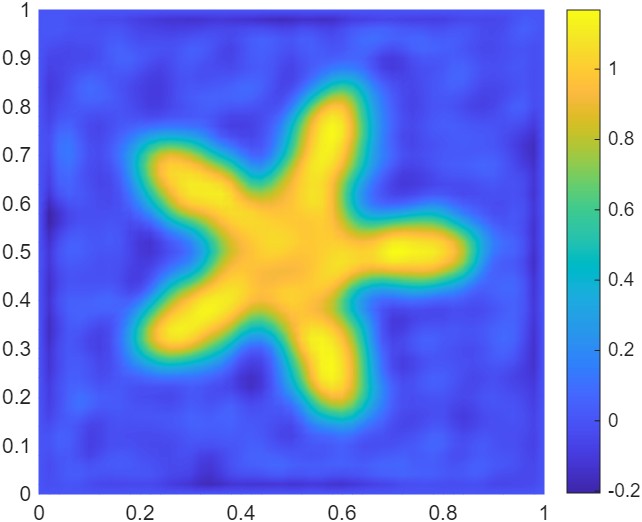}
        \caption{Recovered result by the POD with 2\% noise} % 子图的独立标题
        \label{star_pod} % 子图的独立标签
    \end{subfigure}
    \vspace{6mm}
    % 第一个子图
    % \hfill
    % \begin{subfigure}[t]{0.3\textwidth} 
    %     \vspace{0pt} 
    %     \centering
    %     \includegraphics[width=\textwidth]{pictures/star.jpg}
    %     \caption{Exact source term} 
    %     % \label{star_exact} 
    % \end{subfigure}
    % \hfill % 在两个子图之间添加水平间距，使其左右对齐
    % 第二个子图
    \begin{subfigure}[t]{0.3\textwidth}  % 同上，保证两个子图加起来不超过textwidth并留有间隙
    \vspace{0pt} 
        \centering
        \includegraphics[width=\textwidth]{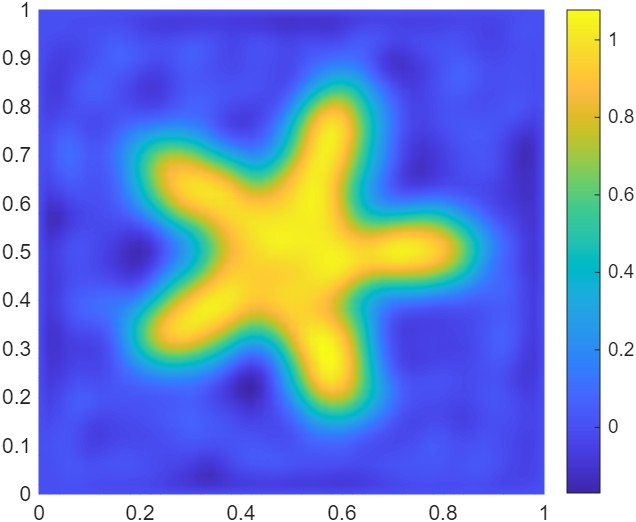}
        \caption{Recovered result by the FEM with 5\% noise} % 子图的独立标题
        % \label{star_fem} % 子图的独立标签
    \end{subfigure}
    \hfill % 在两个子图之间添加水平间距，使其左右对齐
    % \hspace{0.08\textwidth} 
    % 第三个子图
    \begin{subfigure}[t]{0.3\textwidth}  % 同上，保证两个子图加起来不超过textwidth并留有间隙
    \vspace{0pt} 
        \centering
        \includegraphics[width=\textwidth]{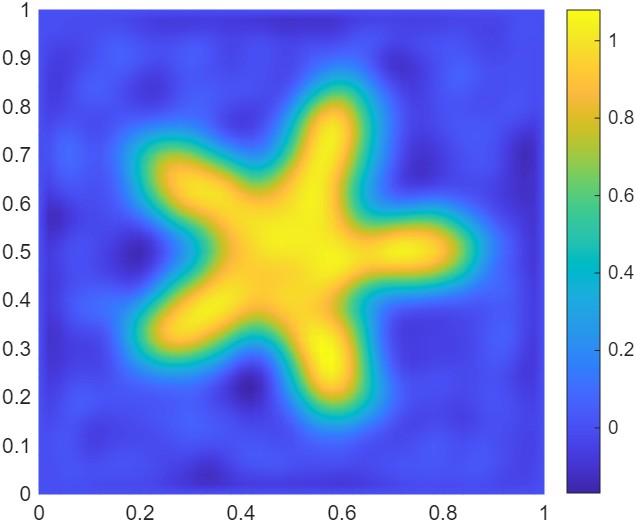}
        \caption{Recovered result by the POD with 5\% noise} % 子图的独立标题
        % \label{star_pod} % 子图的独立标签
    \end{subfigure}
    \hfill 
    \hspace*{0.3\textwidth} 
    % 整个图的共同标题
    \caption{Comparison of FEM and POD reconstruction results for $f^*$ of star-shaped source in \Cref{example_elliptic_pod_performance}}
    \label{star_compar_effectiveness_pod_fem} % 整个图的共同标签
    \end{figure}

    \begin{table}[ht]
    \caption{ Comparison of computational time between FEM and POD for $f^*$ of star-shaped source in \Cref{example_elliptic_pod_performance}}
    \centering
    % \footnotesize
    % \normalsize 
    % \scalebox{1.2}{
    \begin{tabular}{lcccc} 
    \toprule % 加粗顶部横线（1.5pt 线宽）
       Method  & Noise-level   & Off. (s)                    & Opti. (s)     & Total (s)       \\ \midrule
      FEM         &    2\%             &-- & 126.544 & 126.544  \\ \midrule
      POD        &   2\%             & 3.316 &14.763   & 18.079  \\ \midrule
      FEM        &   5\%             & -- &106.264   & 106.264 \\ \midrule
      POD        &   5\%             & 3.288 &14.898   & 18.186  \\
      \bottomrule
    \end{tabular}
    \label{Tab_star_fem_pod}
    \end{table}

    % \begin{figure}[htbp]
    % \centering
    % % 第一个子图
    % \begin{subfigure}[t]{0.3\textwidth} 
    %     \vspace{0pt} 
    %     \centering
    %     \includegraphics[width=\textwidth]{pictures/smooth_bump.jpg}
    %     \caption{Exact source term} 
    %     \label{smooth_bump_exact} 
    % \end{subfigure}
    % \hfill % 在两个子图之间添加水平间距，使其左右对齐
    % % 第二个子图
    % \begin{subfigure}[t]{0.3\textwidth}  % 同上，保证两个子图加起来不超过textwidth并留有间隙
    % \vspace{0pt} 
    %     \centering
    %     \includegraphics[width=\textwidth]{pictures/smooth_bump_pod_noise_01.jpg}
    %     \caption{Recovered result for 1\% noise} % 子图的独立标题
    %     \label{smooth_bump_noise_02} % 子图的独立标签
    % \end{subfigure}
    % \hfill % 在两个子图之间添加水平间距，使其左右对齐
    % % 第三个子图
    % \begin{subfigure}[t]{0.3\textwidth}  % 同上，保证两个子图加起来不超过textwidth并留有间隙
    % \vspace{0pt} 
    %     \centering
    %     \includegraphics[width=\textwidth]{pictures/smooth_bump_pod_noise_03.jpg}
    %     \caption{Recovered result for 3\% noise} % 子图的独立标题
    %     \label{smooth_bump_noise_05} % 子图的独立标签
    % \end{subfigure}
    % % 整个图的共同标题
    % \caption{Recovering multi-peak Gaussian source with different noise level for \Cref{example_elliptic_pod_performance}}
    % \label{smooth_bump_compar_noise_pod} % 整个图的共同标签
    % \end{figure}

\end{myexample}

\begin{myexample}\label{exam_relative_noise_star}
Given that the true source is unavailable in inverse problems, we employ the relative $L^2$-norm stopping criterion given by $\| f^{i+1} - f^{i}\|_{L^2(\Omega)} / \| f^{i+1} \|_{L^2(\Omega)} \leq 1.0 \times 10^{-5}$.
To evaluate the proposed method, we consider a star-shaped source. For a $1$\% noise level, the reconstruction utilizes $7$ POD basis functions  with the regularization parameter $\lambda = 3.9042 \times 10^{-8}$. The POD reconstruction is illustrated in \Cref{star_noise_relative_pod_result}, and the associated computational costs are summarized in \Cref{star_noise_relative_fem_pod_time_result}.

\begin{figure}[htbp]
    \centering
    
    % 第一个 minipage 放图片
    % 使用 [t] 顶部对齐，宽度放大到 0.45 防止拥挤
    \begin{minipage}[t]{0.35\textwidth}
        \vspace{0pt} % 关键排版魔法：强制定义顶部基线
        \centering
        % 如果觉得图片太大，可以把 \textwidth 改成 0.9\textwidth 等微调
        \includegraphics[width=\textwidth]{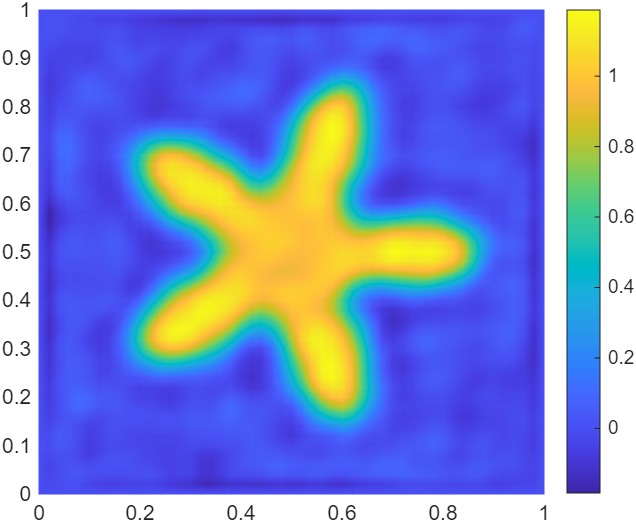} 
        \captionof{figure}{Recovered result by the POD with 1\% noise in \Cref{exam_relative_noise_star}}
        \label{star_noise_relative_pod_result}
    \end{minipage}
    \hspace{0.5cm}
    % \hfill % 使用 hfill 自动推到两边，比固定的 \hspace 更协调
    % 第二个 minipage 放表格
    % 使用 [t] 顶部对齐，宽度放大到 0.48 给长标题充足空间
    \begin{minipage}[t]{0.5\textwidth}
        \vspace{0pt} % 关键排版魔法：强制定义顶部基线
        \centering
        \captionof{table}{Comparison of computational time between FEM and POD for $f^*$ of star-shaped source in \Cref{exam_relative_noise_star}}
        \vspace{2mm} % 让表格和标题之间稍微透点气
        \begin{tabular}{lccc} 
        \toprule 
       Method    & Off. (s)  & Opti. (s) & Total (s) \\ \midrule
      FEM        & --        & 130.648   & 130.648   \\ \midrule
      POD        & 3.382     & 13.078    & 16.460    \\
      \bottomrule
        \end{tabular}
        \label{star_noise_relative_fem_pod_time_result}
    \end{minipage}
    
\end{figure}

\end{myexample}

\begin{myexample}\label{example_elliptic_pod_basis_num}
     
    % Following \Cref{example_lae_elliptic_pod_basis_num}, 
    We investigate how the reduced basis dimension $n$ affects the reconstruction of the $\Omega$-shaped source, with the regularization parameter fixed at  $\lambda = 6.2629 \times 10^{-10}$. 
    % and the penalty norm to $\|\cdot\|^2_X = \|\cdot \|_{L^2(\Omega)}^2$ in problem \cref{opti_problem}. 
    As shown in \Cref{Z_1_b,Z_1_c,Z_1_d,Z_1_e}, increasing $n$ from $1$ to $12$ progressively improves the reconstruction from a coarse approximation of macroscopic features to a high-fidelity recovery of fine-scale details. Quantitatively, the reconstruction error $\| f_{pod} - f^*\|_{L^2(\Omega)}$ plotted in \Cref{Z_1_f} decays rapidly and monotonically as $n$ increases. This trend confirms the excellent approximation performance and numerical convergence of the proposed POD-based approach for complex source identification.

    \begin{figure}[htbp]
        \centering
    % 第一个子图
    \begin{subfigure}[t]{0.3\textwidth} 
        \vspace{0pt} 
        \centering
        \includegraphics[width=\textwidth]{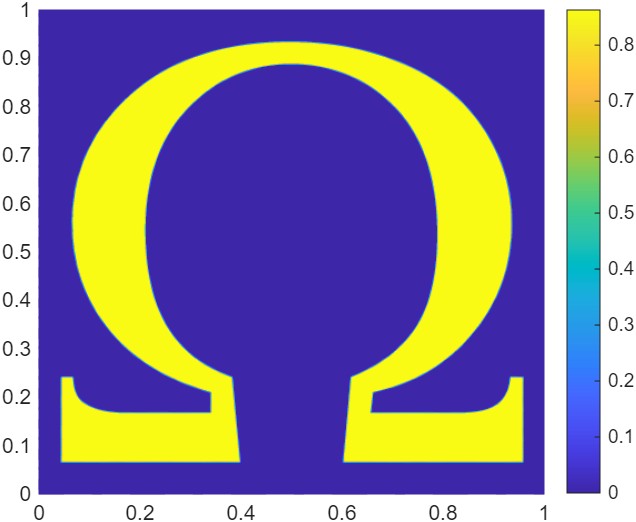}
        \caption{Exact source term} 
        \label{Z_a} 
    \end{subfigure}
    \hfill % 在两个子图之间添加水平间距，使其左右对齐
    % 第二个子图
    \begin{subfigure}[t]{0.3\textwidth}  % 同上，保证两个子图加起来不超过textwidth并留有间隙
    \vspace{0pt} 
        \centering
        \includegraphics[width=\textwidth]{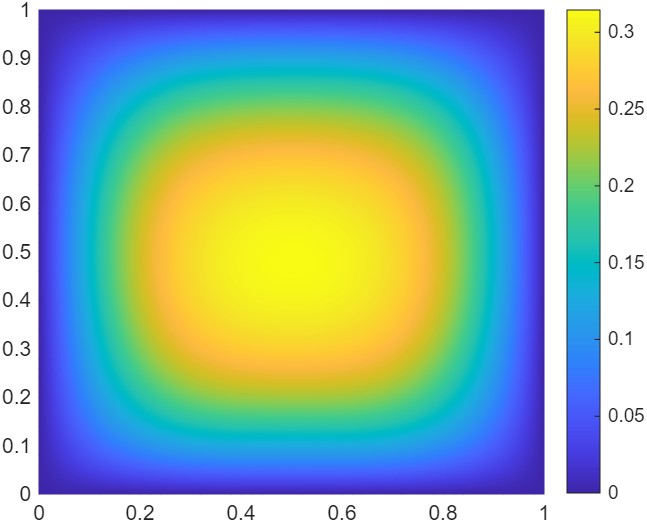}
        \caption{1 POD basis} % 子图的独立标题
        \label{Z_1_b} % 子图的独立标签
    \end{subfigure}
    \hfill % 在两个子图之间添加水平间距，使其左右对齐
    % 第三个子图
    \begin{subfigure}[t]{0.3\textwidth}  % 同上，保证两个子图加起来不超过textwidth并留有间隙
    \vspace{0pt} 
        \centering
        \includegraphics[width=\textwidth]{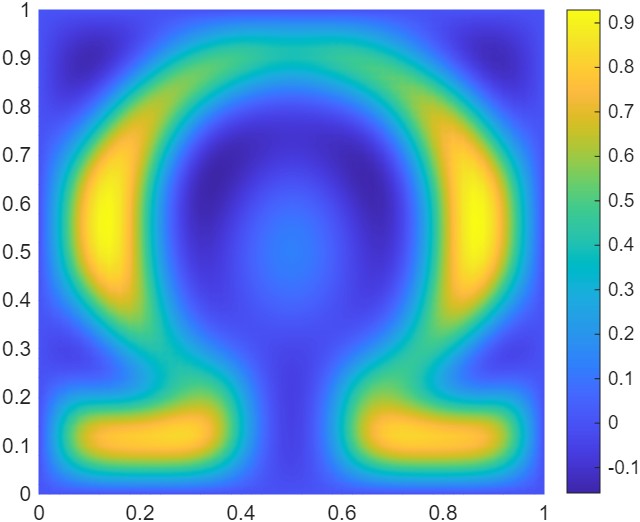}
        \caption{4 POD basis} % 子图的独立标题
        \label{Z_1_c} % 子图的独立标签
    \end{subfigure}
    % 第四个子图
    \begin{subfigure}[t]{0.3\textwidth}  % 同上，保证两个子图加起来不超过textwidth并留有间隙
    \vspace{0pt} 
        \centering
        \includegraphics[width=\textwidth]{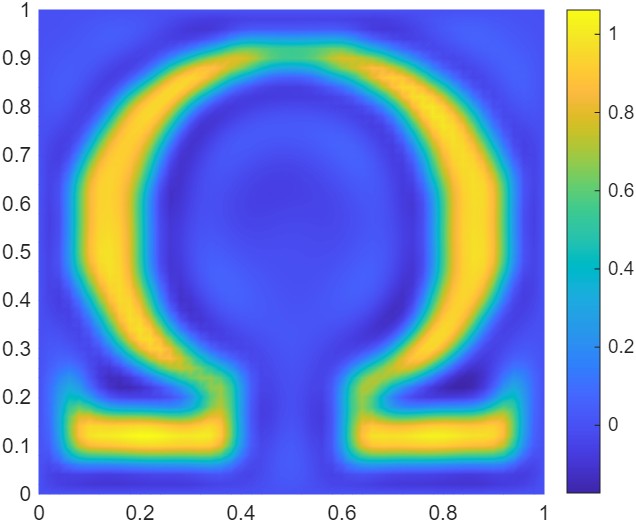}
        \caption{8 POD basis} % 子图的独立标题
        \label{Z_1_d} % 子图的独立标签
    \end{subfigure}
    \hfill
    % 第5个子图
    \begin{subfigure}[t]{0.3\textwidth}  % 同上，保证两个子图加起来不超过textwidth并留有间隙
    \vspace{0pt} 
        \centering
        \includegraphics[width=\textwidth]{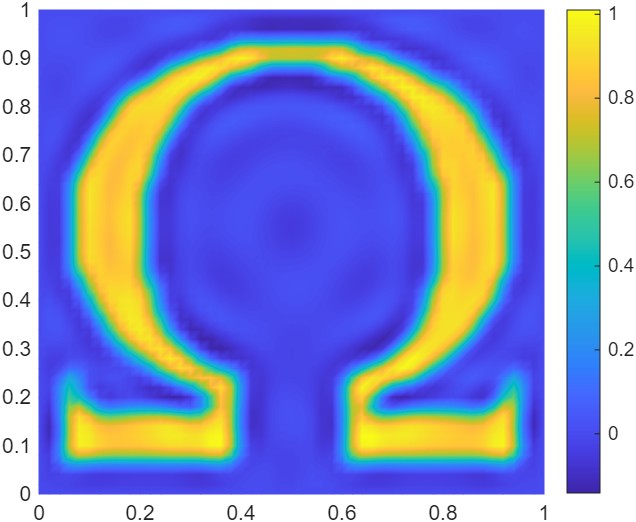}
        \caption{12 POD basis} % 子图的独立标题
        \label{Z_1_e} % 子图的独立标签
    \end{subfigure}
    \hfill
    % 第6个子图
    \begin{subfigure}[t]{0.3\textwidth}  % 同上，保证两个子图加起来不超过textwidth并留有间隙
    \vspace{0pt} 
        \centering
        \includegraphics[width=\textwidth]{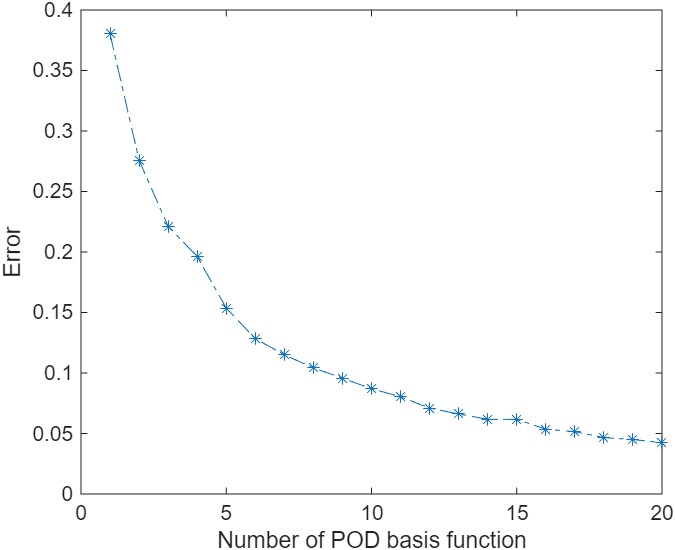}
        \caption{Error} % 子图的独立标题
        \label{Z_1_f} % 子图的独立标签
        \end{subfigure}
    % 整个图的共同标题
        \caption{Effects of the numbers of POD basis for \Cref{example_elliptic_pod_basis_num}}
        \label{Z_pod_number} % 整个图的共同标签
    \end{figure}
\end{myexample}

\subsection{Fredholm Integral Equations of the First Kind}
In this section, we consider the reconstruction of the source $f$ in a Fredholm integral equation of the first kind:
\begin{equation}\label{num_int_eq}
    (\mathcal{K}f)(y) := \int_{\Omega} K(y, x)f(x)dx = m(y), \quad y \in \Omega = (0, 1)^2,
\end{equation}
where we employ a Gaussian kernel $K(y, x) = \exp(-\|x - y\|^2 / \sigma^2)$ with $\sigma=0.05$.
% where $K(y, x)$ is a smooth integral kernel. In our numerical experiments, we specifically employ the Gaussian kernel $K(y, x) = \exp(-\|x - y\|^2 / \sigma^2)$ with $\sigma=0.05$. 
To generate the data $m$, the forward problem \eqref{num_int_eq} is discretized via the FDM on a $50 \times 50$ uniform mesh with grid size $h = 1/49$, and the solution is interpolated onto a $600 \times 600$ uniform sensor grid.
Following \Cref{sec_FIEFK}, the reduced basis is constructed using 400 snapshots from equation \eqref{eq_int_pod_basis_generate} over $(0, 1]$. These snapshots are computed using the same FDM spatial discretization and a backward Euler scheme with
with a time step $\Delta t = 1/400$. For the optimization problem \cref{opti_problem}, we define the penalty norm as $\|\cdot\|^2_X = \|\cdot \|_{L^2(\Omega)}^2$ and terminate the gradient descent when the error satisfies $\| f_{fdm} - f^*\|_{L^2(\Omega)}^2 \leq \varepsilon$ (or $\| f_{pod} - f^*\|_{L^2(\Omega)}^2 \leq \varepsilon$) for a prescribed tolerance  $\varepsilon$.

\begin{myexample}\label{int_eq_noise_performance}
    To evaluate the robustness of the proposed algorithm against data perturbations, we consider a C-shaped source. Reconstructions are performed under $1$\% and $2$\% noise levels using $10$ POD basis functions. For the $1$\% noise case, the regularization parameter is set to $\lambda = 6.3301 \times 10^{-7}$  with a tolerance of $\varepsilon = 2.32\times 10^{-2}$. For the $2$\% noise case, these parameters are adjusted to $\lambda = 2.3135 \times 10^{-6}$ and $\varepsilon = 3.48\times 10^{-2}$. As demonstrated in \Cref{C_compar_effectiveness_pod_fem}, the algorithm exhibits strong stability, accurately recovering the source structure despite the noise. The associated computational costs are summarized in \Cref{Tab_C_fem_pod}.
    
    % We evaluate the robustness of the proposed algorithm against data perturbations. Using the C-shaped source,
    % $f=sin(2\pi x) sin(2\pi y)$ and 15 POD basis functions,  
    % reconstructions are performed under 1\% and 2\% noise levels. The corresponding parameters are $\lambda = 6.3301 \times 10^{-7}$ and  $\lambda = 2.3135 \times 10^{-6}$ respectively, with a fixed tolerance of $\varepsilon = 2.32\times 10^{-2}$ and $\varepsilon = 3.48\times 10^{-2}$. As \Cref{C_compar_effectiveness_pod_fem} demonstrates, the algorithm maintains strong stability, accurately recovering the structure of the source even in the presence of significant noise. The corresponding computational costs are shown in \Cref{Tab_C_fem_pod}.

    \begin{figure}[htbp]
    \centering
    % 第一个子图
    \begin{subfigure}[t]{0.3\textwidth} 
        \vspace{0pt} 
        \centering
        \includegraphics[width=\textwidth]{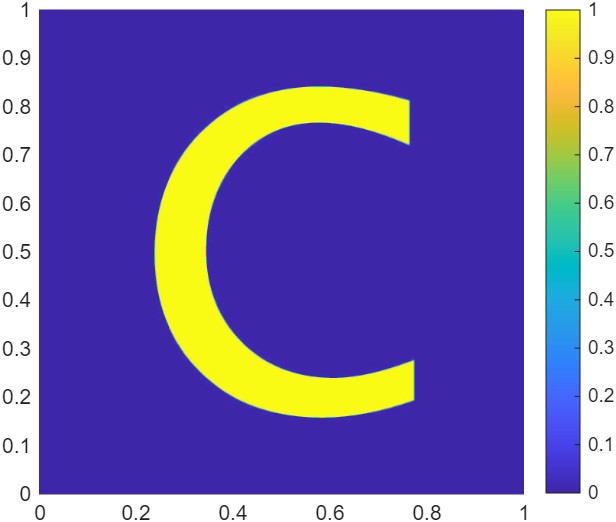}
        \caption{Exact source term} 
        % \label{star_exact} 
    \end{subfigure}
    \hfill % 在两个子图之间添加水平间距，使其左右对齐
    % 第二个子图
    \begin{subfigure}[t]{0.3\textwidth}  % 同上，保证两个子图加起来不超过textwidth并留有间隙
    \vspace{0pt} 
        \centering
        \includegraphics[width=\textwidth]{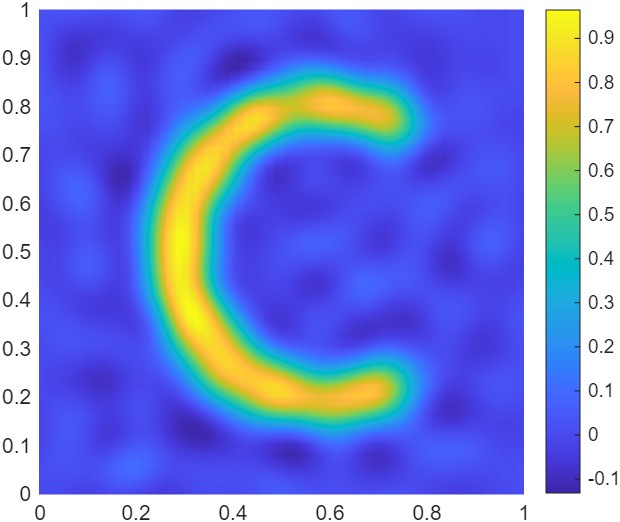}
        \caption{Recovered result by the FEM with 1\% noise} % 子图的独立标题
        % \label{star_fem} % 子图的独立标签
    \end{subfigure}
    \hfill % 在两个子图之间添加水平间距，使其左右对齐
    % 第三个子图
    \begin{subfigure}[t]{0.3\textwidth}  % 同上，保证两个子图加起来不超过textwidth并留有间隙
    \vspace{0pt} 
        \centering
        \includegraphics[width=\textwidth]{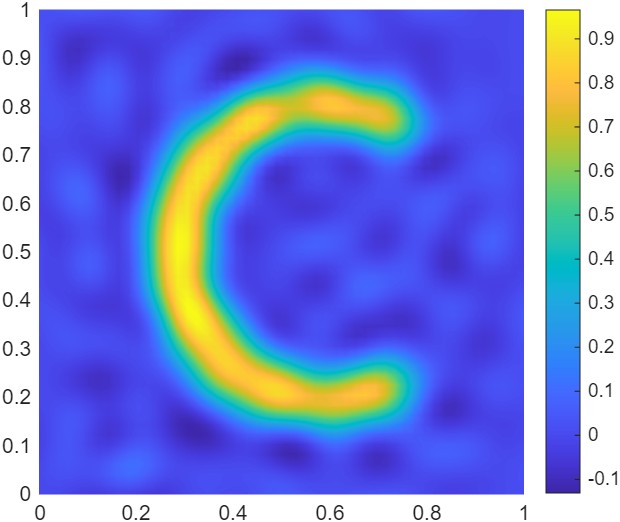}
        \caption{Recovered result by the POD with 1\% noise} % 子图的独立标题
        % \label{star_pod} % 子图的独立标签
    \end{subfigure}
    \vspace{6mm}
    % 第一个子图
    % \hfill
    % \begin{subfigure}[t]{0.3\textwidth} 
    %     \vspace{0pt} 
    %     \centering
    %     \includegraphics[width=\textwidth]{pictures/C.jpg}
    %     \caption{Exact source term} 
    %     % \label{star_exact} 
    % \end{subfigure}
    % \hfill % 在两个子图之间添加水平间距，使其左右对齐
    % 第二个子图
    \begin{subfigure}[t]{0.3\textwidth}  % 同上，保证两个子图加起来不超过textwidth并留有间隙
    \vspace{0pt} 
        \centering
        \includegraphics[width=\textwidth]{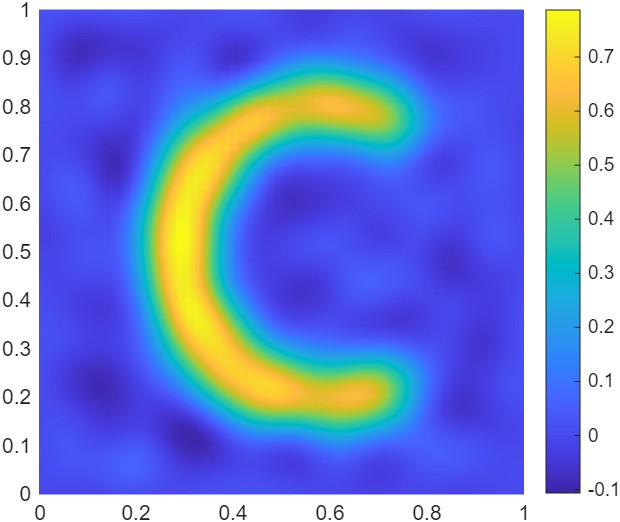}
        \caption{Recovered result by the FEM with 2\% noise} % 子图的独立标题
        % \label{star_fem} % 子图的独立标签
    \end{subfigure}
    \hfill % 在两个子图之间添加水平间距，使其左右对齐
    % 第三个子图
    \begin{subfigure}[t]{0.3\textwidth}  % 同上，保证两个子图加起来不超过textwidth并留有间隙
    \vspace{0pt} 
        \centering
        \includegraphics[width=\textwidth]{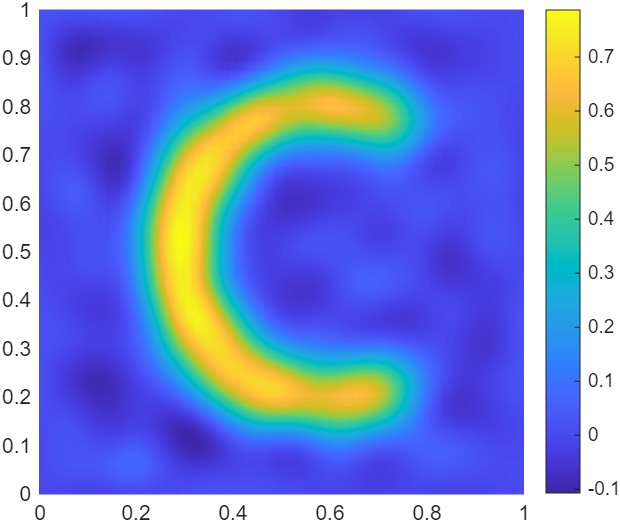}
        \caption{Recovered result by the POD with 2\% noise} % 子图的独立标题
        % \label{star_pod} % 子图的独立标签
    \end{subfigure}
     \hfill 
    \hspace*{0.3\textwidth}
    % 整个图的共同标题
    \caption{Comparison of FEM and POD reconstruction results for $f^*$ of C-shaped source in \Cref{int_eq_noise_performance}}
    \label{C_compar_effectiveness_pod_fem} % 整个图的共同标签
    \end{figure}

    \begin{table}[ht]
    \caption{ Comparison of computational time between FEM and POD for $f^*$ of C-shaped source in \Cref{int_eq_noise_performance}}
    \centering
    % \footnotesize
    % \normalsize 
    % \scalebox{1.2}{
    \begin{tabular}{lcccc} 
    \toprule % 加粗顶部横线（1.5pt 线宽）
       Method  & Noise-level   & Off. (s)                    & Opti. (s)     & Total (s)       \\ \midrule
      FDM         &    1\%             &-- & 243.247 & 243.247  \\ \midrule
      POD        &   1\%             & 2.693 &4.972   & 7.665 \\ \midrule
      FEM        &   2\%             & -- &126.810   & 126.810 \\ \midrule
      POD        &   2\%             & 3.014 &5.556   & 8.570  \\
      \bottomrule
    \end{tabular}
    \label{Tab_C_fem_pod}
    \end{table}

\end{myexample}

\begin{myexample}\label{int_eq_pod_basis_num}
    % Following \Cref{example_elliptic_pod_basis_num},
    In this example, we investigate the influence of the reduced basis dimension $n$ on the reconstruction of the $A$-shaped source with $\lambda = 6.1258 \times 10^{-11}$. As illustrated in \Cref{shan_parabolic_1_b,shan_parabolic_1_c,shan_parabolic_1_d,shan_parabolic_1_e}, increasing $n$ from $1$ to $12$ progressively refines the recovered result, transitioning from a coarse approximation to a high-fidelity representation. This visual improvement is quantitatively corroborated by the rapid and monotonic decay of the reconstruction error $\| f_{pod} - f^*\|_{L^2(\Omega)}$, as depicted in \Cref{shu_parabolic_1_f}.

    % Such striking convergence behavior firmly validates the computational efficiency of the proposed method.
    
    % Such rapid convergence validates the proposed method's efficiency.

    \begin{figure}[htbp]
        \centering
    % 第一个子图
    \begin{subfigure}[t]{0.3\textwidth} 
        \vspace{0pt} 
        \centering
        \includegraphics[width=\textwidth]{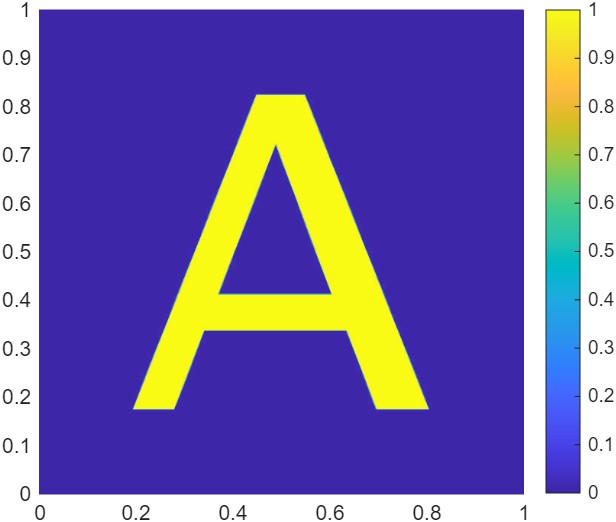}
        \caption{Exact source term} 
        \label{shan_parabolic_a} 
    \end{subfigure}
    \hfill % 在两个子图之间添加水平间距，使其左右对齐
    % 第二个子图
    \begin{subfigure}[t]{0.3\textwidth}  % 同上，保证两个子图加起来不超过textwidth并留有间隙
    \vspace{0pt} 
        \centering
        \includegraphics[width=\textwidth]{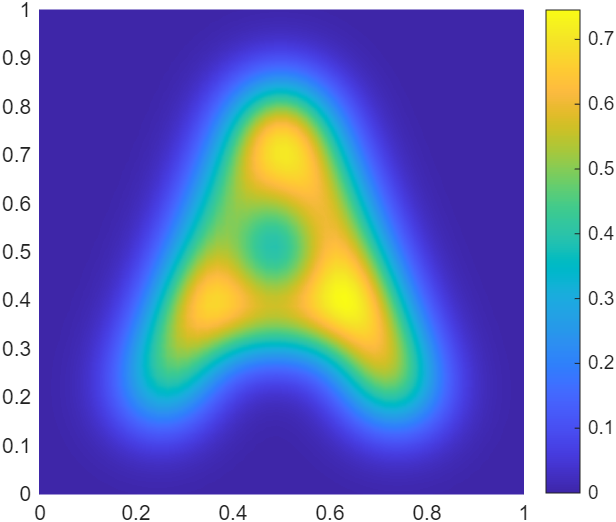}
        \caption{1 POD basis} % 子图的独立标题
        \label{shan_parabolic_1_b} % 子图的独立标签
    \end{subfigure}
    \hfill % 在两个子图之间添加水平间距，使其左右对齐
    % 第三个子图
    \begin{subfigure}[t]{0.3\textwidth}  % 同上，保证两个子图加起来不超过textwidth并留有间隙
    \vspace{0pt} 
        \centering
        \includegraphics[width=\textwidth]{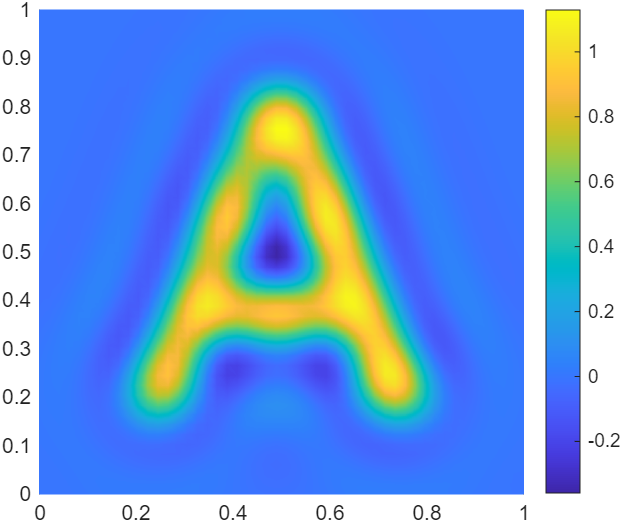}
        \caption{4 POD basis} % 子图的独立标题
        \label{shan_parabolic_1_c} % 子图的独立标签
    \end{subfigure}
    % 第四个子图
    \begin{subfigure}[t]{0.3\textwidth}  % 同上，保证两个子图加起来不超过textwidth并留有间隙
    \vspace{0pt} 
        \centering
        \includegraphics[width=\textwidth]{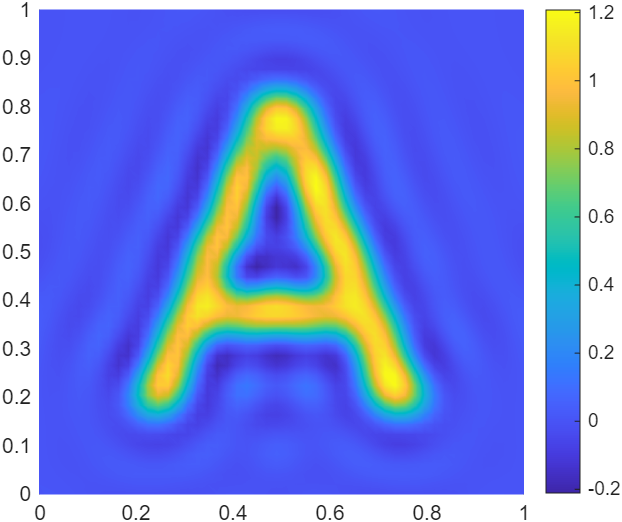}
        \caption{8 POD basis} % 子图的独立标题
        \label{shan_parabolic_1_d} % 子图的独立标签
    \end{subfigure}
    \hfill
    % 第5个子图
    \begin{subfigure}[t]{0.3\textwidth}  % 同上，保证两个子图加起来不超过textwidth并留有间隙
    \vspace{0pt} 
        \centering
        \includegraphics[width=\textwidth]{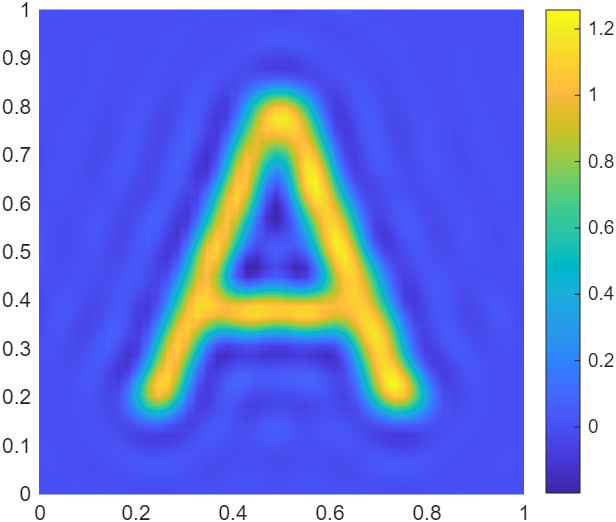}
        \caption{12 POD basis} % 子图的独立标题
        \label{shan_parabolic_1_e} % 子图的独立标签
    \end{subfigure}
    \hfill
    % 第6个子图
    \begin{subfigure}[t]{0.3\textwidth}  % 同上，保证两个子图加起来不超过textwidth并留有间隙
    \vspace{0pt} 
        \centering
        \includegraphics[width=\textwidth]{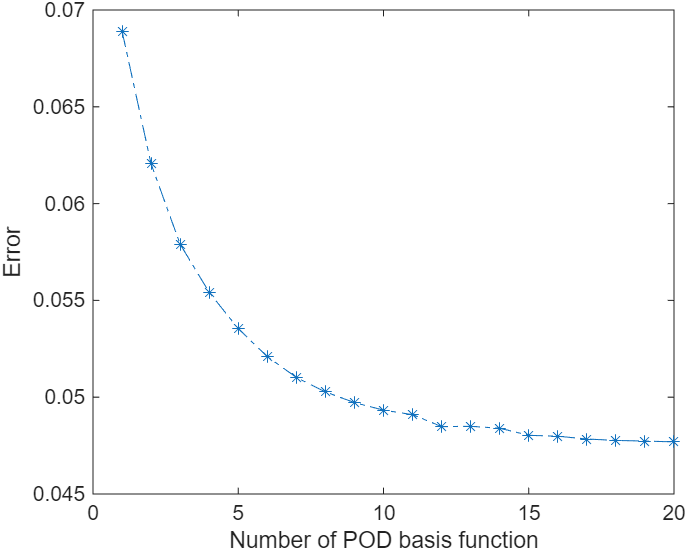}
        \caption{Error} % 子图的独立标题
        \label{shu_parabolic_1_f} % 子图的独立标签
        \end{subfigure}
    % 整个图的共同标题
        \caption{Effects of the numbers of POD basis for \Cref{int_eq_pod_basis_num}}
        \label{shan_pod_number_parabolic} % 整个图的共同标签
    \end{figure}
\end{myexample}

% \section{Conclusion}\label{sec_pod_concl}

\section{Conclusion and Extensions}\label{sec_pod_concl}

In this paper, we proposed a pseudo-time data-driven Proper Orthogonal Decomposition (POD) framework for model reduction of stationary linear operator equations lacking natural temporal snapshot data. By reformulating a static operator equation into a pseudo-time evolution problem, the proposed method successfully generates representative snapshots while inherently preserving the spectral structure of the underlying operator.

Theoretical analysis was provided to justify the proposed framework. In particular, we rigorously proved the exponential convergence of the pseudo-time solution to the exact stationary solution and established the approximation properties of the resulting POD basis functions. These results provide a solid mathematical foundation for the reduced-order approximations.

The effectiveness of the method was extensively demonstrated through two representative settings: elliptic inverse source problems and Fredholm integral equations of the first kind. Numerical results consistently show that accurate approximations can be obtained using only a significantly truncated set of POD basis functions. Furthermore, the method achieves substantial online acceleration compared with full-order FEM or FDM discretizations, and exhibits strong robustness against observational noise in inverse problem scenarios.

Despite these advantages, one limitation of the current framework is that the state or observation variable $u$ and the unknown variable $f$ are assumed to be approximated within the same POD space. While this assumption simplifies the construction of the reduced model, it may be restrictive in general problems, since $f$ and $u$ often exhibit different regularity, spatial scales, and physical structures.

To overcome this limitation, a natural extension is to construct separate POD spaces for $u$ and $f$,
\[
    V_{pod}^u=\operatorname{span}\{\psi_k^u\}_{k=1}^{N_u},
    \qquad
    V_{pod}^f=\operatorname{span}\{\psi_k^f\}_{k=1}^{N_f},
\]
and seek reduced approximations $u_r\in V_{pod}^u$ and
$f_r\in V_{pod}^f$. The two bases may be generated from the pseudo-time
systems
\[
\begin{cases}
    \tilde{u}_t+\mathcal{R}^{-1}\tilde{u}=m,\\[2mm]
    \tilde{u}(\cdot,0)=0,
\end{cases}
\qquad
\begin{cases}
    \tilde{f}_t+\mathcal{R}\tilde{f}=m,\\[2mm]
    \tilde{f}(\cdot,0)=0.
\end{cases}
\]

This decoupled construction provides greater flexibility and is expected to better capture the distinct smoothness and structural properties of $f$ and $u$. The rigorous mathematical analysis and numerical validation of this extended formulation will be investigated in our future work.

% \section{Prospect}
% \textbf{Extension}
% We can choose different POD basis functions for $u$ and $f$, respectively.

% We obtain the POD basis functions $V_{pod}^u=\operatorname{span}\{\psi^u_k\}_{k=1}^l$ using the following equation.
% \begin{equation}
%     \begin{aligned}
%         \left\{\begin{array}{ll}
%             \widetilde{u}_{t} + \mathcal{R}^{-1} \widetilde{u} =  m &\text { in }~ \Omega \times (0, T), \\[2mm] 
%             \widetilde{u}_0 = 0 &\text { in }~ \Omega.\end{array}\right.
%         \end{aligned}
% \end{equation}

% And we obtain the $V_{pod}^f=\operatorname{span}\{\psi^f_k\}_{k=1}^l$ by using the following equation:
% \begin{equation}
%     \begin{aligned}
%         \left\{\begin{array}{ll}
%             \widetilde{f}_{t} + \mathcal{R}\widetilde{f} =  m &\text { in }~ \Omega \times (0, T), \\[2mm] 
%             \widetilde{f}_0 = 0 &\text { in }~ \Omega.\end{array}\right.
%         \end{aligned}
% \end{equation}

% Both sets of the aforementioned POD basis functions are utilized to construct a low-rank approximation for the following least-squares problem.
% \begin{equation}
%     \min_{f\in V_{pod}^f}\|\mathcal{R}f-m(x)|\vert_{L^2}^2+\lambda \|f|\vert_{L^2}^2.
% \end{equation}
% The gradient of the above problem is:
% \begin{equation}
%     \mathcal{R}^*\mathcal{R}f+\lambda f=\mathcal{R}^*m(x).
% \end{equation}

% Since $\mathcal{R} f=u$, it follows that:
% \begin{equation}
%     \mathcal{R}^* u_{pod}+\lambda f_{pod}=\mathcal{R}^* m(x) \qquad u_{pod}\in V_{pod}^u,f_{pod}\in V_{pod}^f.
% \end{equation}

\appendix

% \section{Proper Orthogonal Decomposition (POD) Method}\label{sec_pod_method}

% \newpage
\renewcommand\refname{References}
\bibliographystyle{plain}
\bibliography{references}

%\bibliographystyleMath{unsrt}
%\bibliographyMath{refs-etc}

%\bibliographystylePhys{unsrt}
%\bibliographyPhys{refs-etc}

\end{sloppypar}
\end{document}